\documentclass[11pt,letterpaper]{amsart}
\usepackage[T1]{fontenc}
\usepackage[toc, page]{appendix}
\usepackage{amssymb, amsfonts, latexsym, mathrsfs, amsmath, amsthm, bbm, amstext}
\usepackage{txfonts,bbding,wasysym,pifont,stmaryrd,tipa, bm, amsbsy}
\usepackage{tikz-cd}
\usepackage[all,cmtip]{xy}
\usepackage{float, mathtools, enumitem}
\usepackage[utf8]{inputenc}
\usepackage{hyperref}
\usepackage{cleveref}
\usepackage{enumitem}
\crefname{section}{§}{§§}
\Crefname{section}{§}{§§}

\newtheorem{theorem}{Theorem}[section]
\newtheorem{lemma}[theorem]{Lemma}
\newtheorem{proposition}[theorem]{Proposition}
\newtheorem{corollary}[theorem]{Corollary}

\newtheorem{que}[theorem]{Question}

\theoremstyle{definition}
\newtheorem{definition}[theorem]{Definition}

\theoremstyle{remark}
\newtheorem{remark}[theorem]{Remark}
\newtheorem{notation}[theorem]{Notation}
\numberwithin{equation}{section}

\DeclareMathAlphabet{\mathpzc}{OT1}{pzc}{m}{it}

\DeclareMathOperator{\sheafHom}{\mathscr{H}\text{\kern -3pt {\calligra\large om}}\,}

\renewcommand\rightmark{Galois covers of singular curves in positive characteristics}

\begin{document}

\title{Galois covers of singular curves in positive characteristics}
\author{Soumyadip Das}
\address{School of Mathematics,
Tata Institute of Fundamental Research, Mumbai, Maharastra 400005.}
\email{dass@tifr.math.res.in}



\begin{abstract}
We study the \'{e}tale fundamental groups of singular reduced connected curves defined over an algebraically closed field of arbitrary prime characteristic. It is shown that when the curve is projective, the \'{e}tale fundamental group is a free product of the \'{e}tale fundamental group of its normalization with a free finitely generated profinite group whose rank is well determined. As a consequence of this result and the known results for the smooth case, necessary conditions are given for a finite group to appear as a quotient of the \'{e}tale fundamental group. Next, we provide similar results for an affine integral curve $U$. We provide a complete group theoretic classification on which finite groups occur as the Galois groups for Galois \'{e}tale connected covers of $U$. In fact, when $U$ is a seminormal curve embedded in a connected seminormal curve $X$ such that $X - U$ consists of smooth points, the tame fundamental group $\pi_1^t(U \subset X)$ is shown to be isomorphic to a free product of the tame fundamental group of the normalization of $U$ with a free finitely generated profinite group whose rank is known. An analogue of the Inertia Conjecture is also posed for certain singular curves.
\end{abstract}

\maketitle

\section{Introduction}
The main objective of this article to relate the \'{e}tale fundamental group of a reduced (singular) curve defined over an algebraically closed field of positive characteristic with the \'{e}tale fundamental group of its normalization.

Let $p$ be a prime number and $k$ be an algebraically closed field of characteristic $p$. The \'{e}tale fundamental group $\pi_1(X)$ of a smooth projective connected $k$-curve is a finitely generated profinite group, but not much is known about the structure of the group $\pi_1(X)$ itself. As a finitely generated profinite group, its structure is completely determined by its finite quotients. Due to works of Grothendieck, Shafarevich Hasse--Witt, Nakajima, Stevenson and others (see \cite[Section 7]{survey_paper}), we know some necessary conditions on these finite quotients. If $U \subset X$ is the affine curve obtained by removing $r$ closed points of $X$, the \'{e}tale fundamental group $\pi_1(U)$ is no longer finitely generated (in contrast to what happens over $\mathbb{C}$). In this case, we know all the finite quotients by Abhyankar's Conjecture on Affine Curves (\cite[Conjecture~3.2]{survey_paper}, which is now a Theorem due to Serre, Raynaud and Harbater; \cite[Section~3.3]{survey_paper}). More precisely, a finite group $G$ occurs as a quotient of $\pi_1(U)$ if and only if $G/p(G)$ is generated by at most $2 g(X) + r -1$ elements (here $g(X)$ is the genus of $X$, and $p(G)$ is the normal subgroup generated by the Sylow $p$-subgroups of $G$). In general, this group theoretic description does not determine the structure of the group $\pi_1(U)$ as we do not know how they fit into the inverse system that defines the profinite group $\pi_1(U)$.

In this article, we first consider $\pi_1(X)$ for an arbitrary reduced projective connected $k$-curve $X$ that is not necessarily irreducible and may contain singular points. Let $\nu \colon X^\nu \longrightarrow X$ be the normalization map. Then we have the following structure theorem for $\pi_1(X)$ in terms of $\pi_1(X^\nu)$ and the number of points in the fibres of the normalization map $\nu$.

\begin{theorem}[{Theorem~\ref{thm_main_projective}}]\label{intro_thm_proj}
Let $k$ be an algebraically closed field of arbitrary characteristic. Let $X$ be a connected projective $k$-curve having $n$ irreducible components. Consider the normalization map $\nu \colon X^\nu \longrightarrow X$. Set
$$\delta \coloneqq 1 - n + \sum_{x \in X} \left( |\nu^{-1}(x)| -1 \right).$$
We have an isomorphism of finitely generated profinite groups:
$$\pi_1(X) = \pi_1(X^\nu) \ast \widehat{F_\delta}$$
where $\widehat{F_\delta}$ is the profinite completion of a free group $F_\delta$ on $\delta$ generators.
\end{theorem}

Note that the summation in the above definition of $\delta$ can be taken over the singular points of $X$. To prove the result, we first reduce to the case where $X$ is a seminormal curve (Lemma~\ref{lem_same_pi_for_semi-normal}). As a consequence of \cite{Kaya}, the lifting problem for coves of $X$ is solvable (Proposition~\ref{prop_lift_seminormal}). It follows that there is a complex projective connected seminormal curve $X_{\mathbb{C}}$ satisfying: (1) $X_{\mathbb{C}}$ has $n$ irreducible components; (2) The above defined number $\delta$ for $X_{\mathbb{C}}$ is the same as $\delta$ for $X$; (3) (see \eqref{eq_surj}) there is a surjective specialization map $\text{\rm sp} \colon \pi_1(X_{\mathbb{C}}) \twoheadrightarrow \pi_1(X)$. We have a complete description of $\pi_1(X_{\mathbb{C}})$ by the Riemann Existence Theorem (\cite[Exp XII; Corollary~5.2]{SGA1}) and a topological argument. This shows that $\pi_1(X)$ is a finitely generated profinite group. Further using the surjectivity of the above map $\text{\rm sp}$ and the specialization map in the smooth case, we establish that $\pi_1(X)$ is a quotient of $\pi_1(X^\nu) \ast \widehat{F_\delta}$. We then prove that there is also a surjection $\pi_1(X) \twoheadrightarrow \pi_1(X^\nu) \ast \widehat{F_\delta}$. To see this, we first note that (Remark~\ref{rmk_factorization}) $X$ is obtained from $X^\nu$ by identifying two points at a time. Then the result follows by inductively applying Proposition~\ref{prop_main_1} and Proposition~\ref{prop_main_2}. Roughly, these propositions allow us to construct connected Galois covers of $X$ starting from connected Galois covers of $X^{\nu}$. Here we suitably identify points of the induced covers to obtain the final cover. This is managed by use of Lemma~\ref{lem_cover_by_identification}. Let us mention that the above propositions preserve the ramification behavior of the covers (Remark~\ref{rmk_preserving_inertia}), and up to some finitely many choices involving the coset representatives and finitely many points, this method canonically constructs all the possible connected Galois covers of $X$ from those of $X^\nu$ (see Remark~\ref{rmk_categorical}).

As an application of Theorem~\ref{intro_thm_proj} together with the knowledge from the smooth case, we easily establish some important necessary conditions for a finite group $G$ to be realized as a quotient of the group $\pi_1(X)$ (see Corollary~\ref{cor_nec}).

Next, we consider the case when $U$ is a connected affine singular $k$-curve. For simplicity, we also assume that $U$ is irreducible. Our first main result gives a a complete group theoretic classification for finite (continuous) quotients of $\pi_1(U)$, analogous to the smooth case. Recall that for a finite group $G$, (necessarily normal, quasi $p$-subgroup) $p(G)$ is the subgroup generated by the Sylow $p$-subgroups of $G$.

\begin{theorem}[{Theorem~\ref{thm_main_affine}}]\label{intro_thm_affine}
Let $k$ be any algebraically closed field of arbitrary characteristic and $U$ be an integral affine $k$-curve. Let $X^\nu$ be the unique (up to isomorphism) smooth projective connected $k$-curve with function field $k(X^\nu) = k(U)$. Let $U^\nu \subset X^\nu$ be the maximal affine open subset together with a finite surjective morphism $\nu \colon U^{\nu} \longrightarrow U$ that is an isomorphism over the smooth points of $U$. Suppose that $X^\nu$ has genus $g$, \, $r = |X^\nu - U^\nu|$ and set
$$\delta \coloneqq \sum_{u \in U} \left( |\nu^{-1}(u)| -1 \right).$$

A finite group $G$ occurs as a (continuous) quotient of the \'{e}tale fundamental group $\pi_1(U)$ if and only if $G/p(G)$ is generated by at most $2g+r-1+\delta$ elements.
\end{theorem}

As in the projective case, we first reduce the problem to: $U$ is seminormal. Additionally, we can always embed $U$ in a seminormal irreducible curve $X$ so that $X - U$ contains $r$ smooth points. We then see (as a consequence of a local-global principle in Proposition~\ref{prop_loc_glob}) that for a finite (continuous) group quotient $\pi_1(U) \twoheadrightarrow G$, the prime-to-$p$ quotient group $G/p(G)$ can be realized as the Galois group of a connected cover of $U_{\mathbb{C}}$, where $U_{\mathbb{C}}$ is a complex connected affine curve obtained by removing $r$ smooth points from the curve $X_{\mathbb{C}}$ described earlier. A complete (well known) realization of $\pi_1(U_{\mathbb{C}})$ then produces the necessary condition on $G$. Conversely, for any $G$ satisfying the condition, we construct a cover of $X$ that is \'{e}tale over $U$. This is done using Proposition~\ref{prop_main_1} and Proposition~\ref{prop_main_2} together with formal patching results.

Additionally (under the notation of Theorem~\ref{intro_thm_affine}), using arguments already developed for the affine and the projective cases, we see that the tame fundamental group $\pi_1^t(U \subset X)$ for the above seminormal curves $U \subset X$ can be related with the usual tame fundamental group $\pi_1^t(U^\nu)$:
\begin{enumerate}
\item[] Theorem~\ref{thm_tame}: \hspace{.2cm} $\pi_1^t(U \subset X) = \pi_1^t(U^\nu) \ast \widehat{F_{\delta}}$.
\end{enumerate}
 
Finally, we pose Question~\ref{que} similar to the generalization of the Inertia Conjecture (see \cite[Section~6]{Das}). As there is not much known about the status of it, we do not venture into the details. Remark~\ref{rmk_last} shows an immediate consequence of Theorem~\ref{thm_main_projective} and formal patching to answer a weaker version of the general question.

The structure of the paper is as follows. Section~\cref{sec_prelim} consists of the definitions and notation we will need; it is fairly self contained. The notion of obtaining curves by identifying effective divisors of a normal curve is recalled and studied briefly in Section~\cref{sec_RS_curve}. The seminormal curves and their lifting problems are the content of Section~\ref{sec_lift}. Our main results on the \'{e}tale fundamental groups are discussed in Section~\cref{sec_Main_proj} and Section~\ref{sec_Main_Affine}. The consequence of these results are in Section~\cref{sec_nec} and in Section~\cref{sec_IC}.

\subsection*{Acknowledgements}
I am indebted to Najmuddin Fakhruddin for discussions and insight; especially for suggesting the statement of our main theorem in the projective case and the use of seminormal curves.

\section{Preliminaries}\label{sec_prelim}
\subsection{Notation, Convention and Definitions}\label{sec_notation}
We use the following notation and convention throughout this article.
\begin{enumerate}
\item We will work over an algebraically closed field $k$; our main interest is when $\text{\rm char}(k)$ is an arbitrary prime $p$. In this case, we let $R$ be a complete discrete valuation ring with residue field $k$ and fraction field $K = \text{\rm Fr}(R)$.
\item All the $k$-curves are assumed to be reduced.
\item An $R$-curve is a reduced scheme $X_R \longrightarrow \text{\rm Spec}(R)$ that is a flat, separated scheme of finite type over $\text{\rm Spec}(R)$ of relative dimension one whose generic geometric fibres are connected. For an $R$-curve $X_R$ and any ring homomorphism $R \rightarrow S$, we denote the base change $X_R \times_R S$ over $\text{\rm Spec}(S)$ by $X_R$. In particular, the special fibre $X_R \times_R k$ of $X_R$ is denoted by $X_k$ and the generic fibre $X_R \times_R K$ by $X_K$.
\item For a finite group $G$, let $p(G)$ denote the (necessarily normal) subgroup of $G$ generated by all the Sylow $p$-subgroups of $G$. A finite group $G$ is said to be a \textit{quasi {$p$}-group} if $p(G) = G$.
\item For any scheme $X$ and a point $x \in X$, the local ring at $x$ is denoted by $\mathcal{O}_{X,x}$. We denote its completion at the maximal ideal by $\widehat{\mathcal{O}}_{X,x}$. When $\widehat{\mathcal{O}}_{X,x}$ is a domain, $K_{X,x}$ stands for its fraction field.
\item For any reduced scheme $X$, we denote its normalization as $\nu \colon X^\nu \longrightarrow X$.
\end{enumerate}

Let $X$ be a reduced scheme. A \textit{cover} of $X$ is defined to be a finite, generically separable morphism $Y \longrightarrow X$. Let $G$ be a finite group. A $G$-\textit{Galois cover} is a cover $Y \longrightarrow X$ of reduced schemes together with an inclusion $\rho \colon G \hookrightarrow \text{\rm Aut}_X(Y)$ such that $G$ acts simply transitively on each generic geometric fibre. Note that the above inclusion is always an isomorphism if $X$ is an integral scheme.

Let $S$ be a locally Noetherian connected reduced scheme. Let $a \colon \text{\rm Spec}\left( \Omega \right) \longrightarrow S$ be any geometric point. The \'{e}tale fundamental group $\pi_1(S,a)$ of $S$ at $a$ was defined in \cite[Exp. V, Section~7, Page 140]{SGA1}. The category $\text{FEt}_S$ of finite \'{e}tale covers of $S$ is a Galois category, equipped with a fibre functor $F \colon \text{FEt}_S \longrightarrow \left( \text{ finite sets} \right)$ that to a finite \'{e}tale cover $T \longrightarrow S$ assigns the set of geometric points $\text{\rm Spec}(\Omega) \longrightarrow T$ lying over $a$. Then $\pi_1(S,a) = \text{\rm Aut}(F)$, and $F$ defines an equivalence of the category $\text{\rm FEt}_S$ of finite \'{e}tale covers of $S$ with $\text{\rm Finite-}\pi_1(S,a)\text{\rm -sets}$. For a different choice of generic point $a'$, there is an automorphism $\pi_1(S,a) \cong \pi_1(S,a')$, canonical up to inner automorphisms. As our interest is the structure of group $\pi_1(S,a)$ itself, we drop the geometric point from notation and only consider the group $\pi_1(S)$, up to isomorphisms. When $S$ is an integral normal scheme, there is an ind-scheme $\widetilde{S}$, \'{e}tale over $S$, and the group $\pi_1(S)$ can be canonically identified with $\text{\rm Aut}\left( \widetilde{S}/S \right)$. In general, $\pi_1(S)$ is a profinite group that is a limit of its finite quotients. Using the properties of a Galois category, we have
$$\pi_1(S) = \underset{\substack{T  \longrightarrow S\\ \text{ finite Galois \'{e}tale cover}}}{\varprojlim} \text{ \rm Aut}\left( T/S \right).$$
By the general version of Riemann Existence Theorem (\cite[Exp XII, Corollary~5.2, page 337]{SGA1}), when $k = \mathbb{C}$, the \'{e}tale fundamental group $\pi_1(X)$ is the profinite completion of the topological fundamental group $\pi_1^{\text{\rm top}}(X(\mathbb{C}))$. We refer to \cite{survey_paper} for a survey on the structure of $\pi_1(X)$ and its quotients, where $X$ is a smooth $k$-curve. 

A lot of these structure theorems in positive characteristic are obtained by first relating to the \'{e}tale fundamental group of an `analogous' complex curve which is then understood from topology. This leads us to \textit{the lifting problems}. Assume that $\text{\rm char}(k) = p >0$. Let $R$ be a complete discrete valuation ring with residue field $k$ and its fraction field $K = \text{\rm Fr}(R)$ with $\text{\rm char}(K) = 0$. Let $X$ be a connected $k$-curve. A \textit{lift of} $X$ \textit{over} $\text{\rm Spec}(R)$ is an $R$-curve $X_R$ with special fibre $X_R \times_R k = X$. The usual lifting problem asks whether the curve $X$ lifts over $\text{\rm Spec}(R)$. We also have the lifting problem for finite covers of $X$: let $f \colon Z \longrightarrow X$ be a finite cover of connected $k$-curves. A \textit{lift of} $f$ \textit{over} $\text{\rm Spec}(R)$ is a finite cover $f_R \colon Y_R \longrightarrow X_R$ of $R$-curves such that the special fibre $f_R \times_R k \colon Y_R \times_R k \longrightarrow X_R \times_R k$ of the map $f_R$ is $f$. Similarly, there are local lifting problems for a complete discrete valuation ring and its extensions. The solution to such problems are well studied when $X$ is a smooth curve. The usual lifting problem for $X$ and the lifting problem for (Galois) covers of $X$ can always be solved (see \cite[Corollary~22.2, page 145]{Hartshorne_Def} and \cite[Theorem~1.5]{Obus}). We refer to \cite{Obus} and \cite{Obus_2} for a detailed exposition on the lifting problems for smooth curves and on the local lifting problems. In this article, we deal with lifting of singular curves and of their \'{e}tale covers. This problem is open in general (\cite[References and Further Results, page 148]{Hartshorne_Def}). As a consequence of \cite{Kaya}, we will see in Proposition~\ref{prop_lift_seminormal} that the usual and the \'{e}tale lifting problem for a seminormal curve has a solution.

\subsection{Rosenlicht-Serre Curves}\label{sec_RS_curve}
This section is devoted to the singular curves `defined by modulus' in \cite[Chapter IV, Section 4, page 61--62]{Serre_Alg}. Let $k$ be an algebraically closed field of arbitrary characteristic.

Let $X$ be a (not necessarily irreducible) projective $k$-curve and $X_{\text{Sing}} \subset X$ be its singular locus. Let $\nu \colon X^{\nu} \longrightarrow X$ denote the normalization map. Then $X^\nu$ is a smooth projective curve, and $\nu$ is a birational finite morphism that induces an isomorphism $X^{\nu} - \nu^{-1}(X_{\text{Sing}}) \cong X - X_{\text{Sing}}$. For each closed point $x \in X$, we have an inclusion of $k$-algebras (see \cite[Chapter IV, Section 1 and 2, page 58--59]{Serre_Alg})
\begin{equation}\label{eq_inclusion}
\left(\nu_{*}\mathcal{O}_{X^\nu}\right)_x = \bigcap_{\nu(y)=x} \mathcal{O}_{X^\nu,y} \supset k + J_x \supset \mathcal{O}_{X,x} \supset k + \mathfrak{c}_x
\end{equation}
where $J_x$ is the Jacobson radical of the semi-local ring $\left(\nu_{*}\mathcal{O}_{X^\nu}\right)_x$ and $\mathfrak{c}_x$ is the conductor ideal $\text{\rm Ann}\left(\left(\nu_{*}\mathcal{O}_{X^\nu}\right)_x/ \mathcal{O}_{X,x}\right)$. Note that $\mathfrak{c}_x$ is characterized by the fact that this is the largest ideal of $\mathcal{O}_{X,x}$ that is also an ideal of $\left( \nu_* \mathcal{O}_{X^\nu} \right)_x$.

On the other hand, a curve $X$ is constructed from its normalization $X^\nu$ in \cite{Rosenlicht} (see \cite[Chapter IV, Section 3 and 4, page 60--62]{Serre_Alg}). Let $X^\nu$ be a smooth projective $k$-curve. Let $D_1, \, \ldots, \, D_l$ be effective (Weil) divisors on $X^\nu$ with pairwise disjoint support where $D_i = \sum_{y \in \text{\rm Supp}\left(D_i \right)} n_y y$, \, $1 \leq i \leq l$. Suppose that for each $i$, \, $\text{\rm deg}\left( D_i \right) = \sum_{y \in \text{\rm Supp}\left(D_i \right)} n_y \geq 2$. Consider the topological space
$$X = \left( X^\nu - \bigsqcup_{1 \leq i \leq l} \text{\rm Supp}\left( D_i \right)\right) \sqcup \{ x_1,\, \ldots, \, x_l\}$$
equipped with cofinite topology. Let $\nu \colon X^\nu \longrightarrow X$ be the topological quotient map. We endow $X$ with a sheaf of rings $\mathcal{O}_X$ such that
\[
\mathcal{O}_{X,x} = \begin{cases}
\mathcal{O}_{X^\nu,\nu^{-1}(x)} & \text{\rm if \,} x \neq x_i \, \text{\rm for any } 1 \leq i \leq l\\
k+\mathfrak{c}_i & \text{\rm if \,} x=x_i, \, 1\leq i \leq l
\end{cases}
\]
where $\mathfrak{c}_i$ is the ideal of $\bigcap_{y \in \text{\rm Supp}\left( D_i \right)} \mathcal{O}_{X^\nu, y}$ formed by functions $f$ such that $f \equiv 0 \pmod{D_i}$, i.e. $\text{\rm v}_y(f) \geq n_y$ for all points $y \in \text{\rm Supp}\left( D_i \right)$, \, $\text{\rm v}_y$ being the valuation attached to $y$. By \cite[Chapter IV, Proposition 2, page 60]{Serre_Alg}, the sheaf $\mathcal{O}_X$ endows $X$ with a structure of an algebraic curve having $\nu \colon X^\nu \longrightarrow X$ as the normalization map. Moreover, (loc. cit. Section 4) the singular locus of $X$ is given by $X_{\text{\rm Sing}} = \{x_1, \, \ldots, \, x_l\}$ and $\mathfrak{c}_i$ is the conductor ideal $\mathfrak{c}_{x_i} = \text{\rm Ann}\left( \left( \nu_*\mathcal{O}_{X^\nu}\right)_{x_i}/\mathcal{O}_{X,x_i} \right)$.

We have the following alternative description for $X$. For any closed point $y \in X^\nu$, let $\mathfrak{m}_y$ denote the maximal ideal of $\mathcal{O}_{X^\nu,y}$. Let $Y^\nu \coloneqq  X^\nu(D_1 + \cdots + D_l)$\footnote{in the notation of \cite[21.7.1, page 277]{EGA}} be the closed subscheme of $X$ defined by the image of $\bigsqcup_{\substack{1 \leq i \leq l}} \bigsqcup_{\substack{y \in \text{\rm Supp}\left( D_i \right)}} \text{Spec}\left( \mathcal{O}_{X^\nu,y}/ \mathfrak{m}_{y}^{n_y} \right)$ under the canonical morphism to $X^\nu$. The ringed space of amalgamated union $X^\nu \bigsqcup_{Y^\nu} \{x_1, \, \ldots, \, x_l \}$ is a scheme, identified with $X$ via the following cocartesian (as well as cartesian; \cite[Theorem~7.1]{Ferrand}) diagram
\begin{center}
\begin{equation*}\label{RS_diagram}
\tag{*}
\begin{tikzcd}
Y^\nu \arrow[hookrightarrow]{r} \arrow{d} & X^\nu \arrow{d}{\nu} \\
\bigsqcup_{1 \leq i \leq l} \text{\rm Spec}\left( k \right) = \{x_1, \ldots, x_l\} \arrow[hookrightarrow]{r} & X
\end{tikzcd}
\end{equation*}
\end{center}
The map $\nu$ is a finite birational morphism inducing isomorphism between $X^\nu - Y^\nu$ and $X - \{x_1, \ldots, x_l\}$, and the natural map $X^\nu \sqcup \{x_1, \ldots, x_l\} \longrightarrow X$ is a surjective morphism that is a quotient map of the underlying topological spaces. As we will be dealing with such curves throughout this article, we make the following (non-standard) definition.

\begin{definition}\label{def_RS_curve}
A curve $X$ is said to be a \textit{Rosenlicht-Serre curve} (an RS curve) if $X$ is obtained as the cocartesian diagram~\eqref{RS_diagram} above from a smooth projective curve $X^\nu$ and a set $\{D_1, \, \ldots, \, D_l\}$ of effective divisors with pairwise disjoint supports. For convenience, we will denote $X = (X^\nu, \{ D_i\}_{1 \leq i \leq l})$.
\end{definition}

Observe that $X$ is projective as $\nu$ is the normalization map and $X^\nu$ is projective. In general, an RS curve need not be connected. Such a curve $X$ is connected if and only if the set $\bigcup_{1 \leq i \leq l} \text{\rm Supp}\left( D_i \right)$ contains at least one point from each irreducible component of $X^\nu$.

As every RS curve is defined using a cocartesian diagram (which are also cartesian), we define admissible morphisms of RS curves as the morphisms of cocartesian diagrams.

\begin{definition}\label{def_morphism_RS}
Suppose that $Z = (Z^\nu, \{E_j\}_{j \in J})$ and $X = (X^\nu, \{D_i\}_{i \in I})$ be two RS curves. An \textit{admissible morphism} $Z \longrightarrow X$ is defined to be a pair $(f^\nu, f_{ \rm sing})$ where $f^\nu \colon Z^\nu \longrightarrow X^\nu$ is a finite morphism of smooth projective $k$-curves, $f_{\rm sing} \colon Z_{ \rm sing} = \{z_j\}_{j \in J} \longrightarrow \{x_i\}_{i \in I} = X_{ \rm sing}$ is a (necessarily finite) surjective morphism of the corresponding singular loci satisfying the following property.
\begin{itemize}
\item For each $j \in J$, \, $f$ induces a finite morphism $Z^\nu(E_j) \longrightarrow X^\nu(D_i)$ of schemes where $f_{ \rm sing}(z_j) = x_i$.
\end{itemize}
Here $X^\nu(D_i)$ (respectively, $Z^\nu(E_j)$) is the closed subscheme of $X^\nu$ (respectively, of $Z^\nu)$ corresponding to the effective divisor $D_i$ (respectively, $E_j$); see \cite[21.7.1, page 277]{EGA}.
\end{definition}

By the universal property of a cocartesian diagram, we see that each admissible morphism $(f^\nu, f_{ \rm sing})$ of RS curves defines a unique finite surjective morphism $f$ of the corresponding singular curves. Moreover, if $f^{\nu}$ is a finite cover, so is $f$. Conversely, suppose that $f \colon Z \longrightarrow X$ be a finite surjective morphism of RS curves. Then $f$ induces a finite morphism $f^\nu \colon Z^\nu \longrightarrow X^\nu$. Since the preimage of a smooth point in $X$ consists of smooth points in $Z$, the map $f$ also induces a finite map $f_{ \rm sing} \colon Z_{ \rm sing} \longrightarrow X_{ \rm sing}$ between finite sets of closed points, but this map need not be surjective.  If we additionally assume that either $f$ \textit{is an \'{e}tale cover} or $f$ is \textit{a cover that is \'{e}tale over the singular points of }$X$, then the induced map $f_{ \rm sing}$ is surjective. In this case, by the universal property of cartesian diagrams, we obtain an admissible finite cover $(f^\nu, f_{ \rm sing})$.

\begin{remark}\label{rmk_admissible_map_correspondence}
Note that the association $(f^\nu, f_{ \rm sing}) \mapsto f$ is a bijective correspondence between the finite admissible covers of RS curves and the finite covers of corresponding singular curves inducing surjection of the singular loci.
\end{remark}

\subsection{Lifting of covers of Seminormal Curves}\label{sec_lift}
Let $k$ be an algebraically closed field of arbitrary characteristic. We start by recalling the algebraic notion of seminormal rings (see \cite{Vitulli} for an exposition). For a ring $B$, let $J(B)$ denote the Jacobson radical of $B$. For any point $x \in \text{Spec}(B)$, let $B_x$ denote the localization of $B$ at the prime ideal $\mathfrak{p}_x$ corresponding to the point $x$. We recall the following definition of seminormality due to Traverso.

\begin{definition}[{\cite[Definition~2.4, page 447]{Vitulli}}]\label{def_semi_normal}
Let $A \subset B$ be an extension of rings. The seminormalization ${}^+_B A$ of $A$ in $B$ is defined to be the ring
$${}^+_B A \coloneqq \{b \in B \, | \, b_x \in A_x + J(B_x) \text{ for all } x \in \text{Spec}(A) \}$$ 
For a Mori ring $A$ (i.e., $A$ is a reduced ring $A$ and the integral closure $\bar{A}$ of $A$ in its total ring of fractions is finite as an $A$-module), the seminormalization $A^+$ of $A$ is defined to be the seminormalization ${}^+_{\bar{A}} A$ of $A$ in $\bar{A}$. In this case, $A$ is said to be a seminormal ring if $A = A^+$.
\end{definition}

Let $A$ be a reduced excellent ring. Let $\bar{A}$ be the integral closure of $A$ in its total ring of fractions. Then the canonical ring extension $A \subset \bar{A}$ factors uniquely as $A \subset A^+ \subset \bar{A}$. Moreover, there is no proper \textit{subintegral extension} contained in $A^+ \subset \bar{A}$, i.e. if $A^+ \subsetneq C \subseteq \bar{A}$ so that the associated map $\text{Spec}\left( C\right) \longrightarrow \text{Spec}\left( A^+ \right)$ is a bijection inducing isomorphisms on the residue fields, we have $C = \bar{A}$.

\begin{definition}\label{def_seminormal_curves}
Let $X$ be a (reduced) $k$-curve. We say that $X$ is \textit{seminormal} if for each point $x \in X$, the local ring $\mathcal{O}_{X,x}$ at $x$ is a seminormal ring.
\end{definition}

Now we see that every seminormal projective $k$-curve is an RS curve (Section~\ref{sec_RS_curve}). Let $X$ be any projective $k$-curve and $\nu \colon X^\nu \longrightarrow X$ be the normalization map. As in Equation~\eqref{eq_inclusion}, there is an inclusion of $k$-algebras
$$A_x \coloneqq \left( \nu_* \mathcal{O}_{X^\nu} \right)_x \supset k + J(A_x) \supset \mathcal{O}_{X,x}$$ for each point $x \in X$. Here $A_x = \bigcap_{\nu(y)=x} \mathcal{O}_{X^\nu,y}$ is the integral closure of $\mathcal{O}_{X,x}$ in its total ring of fractions. By Definition~\ref{def_semi_normal}, the seminormalization of $\mathcal{O}_{X,x}$ is the ring
$$\mathcal{O}_{X,x}^+ = \{ a \in A_x \, | \, a_x \in \mathcal{O}_{X,x} + J(A_x) \}.$$
Now suppose that $X$ is a seminormal curve. Then for each point $x \in X$, we have $\mathcal{O}_{X,x} = \mathcal{O}_{X,x}^+ = k + J(A_x)$. As the conductor ideal $\mathfrak{c}_x$ is the largest common ideal of $\mathcal{O}_{X,x}$ and $A_x$, we have $\mathfrak{c}_x = J(A_x)$. So for each point $x \in X$, the pullback $\nu^*(x)$ is a reduced divisor $\nu^*(x) = \nu^{-1}(x) = \sum_{\nu(y)=x} y$ on $X^\nu$, and we identify the seminormal projective $k$-curve $X$ with the amalgamated union

\begin{equation}\label{eq_sminormal_is_RS}
X = X^\nu \sqcup_{\left( \bigsqcup_{x \in X_\text{\rm sing}} \nu^{-1}(x) \right)} {X_{\text{\rm sing}}}.
\end{equation}

Now we study the lifting problems (see Section~\cref{sec_notation}) for seminormal curves. As mentioned earlier, a solution to a lifting problem for an arbitrary singular curve is not known. As a consequence of \cite{Kaya}, we see that the \'{e}tale lifting problem for a seminormal curve has a solution.

\begin{proposition}\label{prop_lift_seminormal}
Let $X$ be a connected projective seminormal $k$-curve. Let $R$ be a complete discrete valuation ring with residue field $k$ and fraction field $K$. Let $X_{\text{\rm sing}} = \{x_1, \cdots, x_l\}$ be the singular locus of $X$, \, $l \geq 1$. Consider the normalization $\nu \colon X^\nu \longrightarrow X$ and a lift $X_R^\nu$ of $X^\nu$ to a smooth projective $R$-curve. Then there is a connected lift $X_R$ of $X$ over $\text{\rm Spec}(R)$ with geometrically connected fibres and a finite morphism $\nu_R \colon X_R^\nu \longrightarrow X_R$ such that the special fibre $\nu_R \times_R k$ is the map $\nu$, and the generic fibre $\nu_K \coloneqq \nu_R \times_R K \colon X_K^\nu  \longrightarrow X_K $ has the following properties for any algebraic closure $\bar{K}$ of $K$.
\begin{enumerate}
\item The curve $X_K^\nu$ is a smooth projective $K$-curve; in particular, the geometric fibre $X^\nu_{\bar{K}} \coloneqq X_K^\nu \times_R \bar{K}$ is a smooth projective $\bar{K}$-curve. The irreducible components of $X^\nu_{\bar{K}}$ are in a genus preserving bijective correspondence with the irreducible components of $X^\nu$.\label{i1}
\item The curve $X_{\bar{K}} \coloneqq X_K \times_K \bar{K}$ is a connected seminormal projective $\bar{K}$-curve.\label{i2}
\item The map $\nu_{\bar{K}} = \nu_K \times_K \bar{K} \colon X^\nu_{\bar{K}} \longrightarrow X_{\bar{K}}$ is the normalization map.\label{i3}
\item There is a bijection $\phi \colon X_{\text{\rm sing}} \longrightarrow \left( X_{\bar{K}} \right)_{\text{\rm sing}}$ such that for each $x \in X_{\text{\rm sing}}$,
$$|\nu^{-1}(x)| = | (\nu_{\bar{K}})^{-1}(\phi(x))|.$$\label{i4}
\end{enumerate}
Moreover, any finite \'{e}tale cover $f \colon Z \longrightarrow X$ of connected $k$-curves lifts  to a finite \'{e}tale cover $f_R$ of $X_R$. The map $f_R$ is Galois with group $G$ if $f$ is so.
\end{proposition}

\begin{proof}
Since $X$ is assumed to be seminormal, it is an RS curve defined by the smooth projective curve $X^\nu$ and the reduced effective divisors $\left\{D_x = \nu^{-1}(x)\right\}_{x \in X_{\text{\rm sing}}}$. As $R$ is strictly henselian, the specialization map $X_R^\nu (R) \longrightarrow X^\nu(k)$ is surjective. For each $y \in X^\nu$ lying over a singular point of $X$, choose a section $s_y \colon \text{\rm Spec}\left( R \right) \longrightarrow X_R^\nu$ corresponding to $y$. Let $Y_R$ denote the closed subscheme of $X_R^\nu$ defined by the sections $s_y$ as above. By \cite[Lemma~3]{Kaya}, we have the following cocartesian diagram
\begin{center}
\begin{equation*}
\begin{tikzcd}
Y_R \arrow[hookrightarrow]{r} \arrow{d} & X_R^\nu \arrow{d}{\nu_R} \\
\bigsqcup_{x \in X_{\text{\rm sing}}} \text{\rm Spec}\left( R \right) \arrow[hookrightarrow]{r} & X_R
\end{tikzcd}
\end{equation*}
\end{center}
where $X_R$ is a lift of $X$ over $\text{\rm Spec}(R)$. Since the special fibre $X$ is connected, the $R$-curve $X_R$ is connected as well. By flatness of $X^\nu_R$, we have \eqref{i1}. The closed subscheme $\nu_R(Y_R)$ in $X_R$ is defined by a positive reduced cycle of codimension one (\cite[21.7.2, page 277]{EGA}). As its closed fibre $X_{\rm sing}$ consists of finitely many points, $\nu_R(Y_R) = \sum_{x \in X_{\rm sing}} x_R$ where $x_R \colon \text{Spec}(R) \longrightarrow X_R$ has special fibre $x \in X_{\rm sing}$.

To see the other properties of the geometric generic fibre, note that we have the cocartesian and cartesian diagram (\cite[Theorem~7.1]{Ferrand}):
\begin{center}
\begin{equation*}
\begin{tikzcd}
Y_{\bar{K}} \arrow[hookrightarrow]{r} \arrow{d} & X_{\bar{K}}^\nu \arrow{d}{\nu_{\bar{K}}} \\
\{x_{\bar{K}}\}_{x \in X_{\text{\rm sing}}} = \bigsqcup_{x \in X_{\text{\rm sing}}} \text{\rm Spec}\left( \bar{K} \right) \arrow[hookrightarrow]{r} & X_{\bar{K}}
\end{tikzcd}
\end{equation*}
\end{center}
Here $Y_{\bar{K}}$ is the geometric generic fibre of $Y_R$ that is the same as the closed subscheme of $X^\nu_{\bar{K}}$ defined by the generic geometric fibres of the sections $s_y$, \, $\nu(y) \in X_{\rm sing}$, and each point $x_{\bar{K}}$ is the generic geometric fibre of $x_R$. So $X_{\bar{K}}$ has the singular locus $\left(X_{\bar{K}}\right)_{\rm sing} = \{x_{\bar{K}}\}_{x \in X_{\rm sing}}$, and $X_{\bar{K}}$ is identified as the amalgamated union $X_{\bar{K}}^\nu \sqcup_{Y_{\bar{K}}} \left( X_{\bar{K}} \right)_{\rm sing}$. Now the properties \eqref{i2}--\eqref{i4} are straight forward to check.

If $f \colon Z \longrightarrow X$ is a finite \'{e}tale cover of connected $k$-curves, a lift $f_R$ of $f$ over $\text{\rm Spec}(R)$ exists by \cite[Proposition~22.5, page 147]{Hartshorne_Def}. Using a general argument by Grothendieck's Existence Theorem (\cite[Theorem~21.2, page 140]{Hartshorne_Def}) as in \cite[Proof of Theorem~3.1]{Obus}, if $f$ is $G$-Galois, then $f_R$ is Galois with group $G$ as well.
\end{proof}

\begin{remark}\label{rmk_lifting}
In the above proposition, we obtained a lift $X_R$ for the seminormal curve $X$. In general, a lift is not unique (when it exists). For example, for a semistable curve $X$, there exists a lift $X'_R$ with smooth generic fibre (see \cite[Theorem~4.2]{Saidi}). For our purpose, we want the generic geometric fibre of the lift to have similar properties as the curve $X$ near the singular points.
\end{remark}

\begin{remark}
Since $X$ is a proper $R$-curve, the specialization functor $\text{\rm FEt}_{X_R} \longrightarrow \text{\rm FEt}_X$ is an equivalence of categories by \cite[\href{https://stacks.math.columbia.edu/tag/0A48}{Lemma 0A48}]{SP}.
\end{remark}

We end this section with a construction of seminormal curves (see \eqref{eq_identify} below) obtained by identifying points (extending the notion of Section~\cref{sec_RS_curve}) and a descent for covers under an identification map (Lemma~\ref{lem_cover_by_identification}). Once again, we work over an algebraically closed field $k$ of arbitrary characteristic.

Let $Y$ be a seminormal curve and ${\sim}$ be an equivalence relation on the closed points of $Y$ such that the set $\mathcal{C}_{{\sim}}$ of non-trivial equivalence classes is finite and non-trivial. In the following, we obtain a seminormal $k$-curve $Y/{\sim}$ together with a finite birational morphism
\begin{equation}\label{eq_identify}
q \colon Y \longrightarrow Y/{\sim}
\end{equation}
that induces an isomorphism restricted to $Y - \bigsqcup_{\substack{y \in C \\C \in \mathcal{C}_{\sim}}} y$, and for any point $y \in C$ ($C \in \mathcal{C}_{{\sim}}$), \, $q^* (q(y))$ is the reduced divisor $\sum_{y' \in C} y'$ on $Y$.

Suppose that $Y_{\rm sing} = \{y_1, \ldots, y_l\}$. Since $Y$ is seminormal, it is an RS curve
$$Y = Y^\nu \sqcup_{\left( \bigsqcup_{1 \leq i \leq l} \nu^{-1}(y_i) \right)} \{ y_1, \ldots, y_l\} \hspace{2 em} \text{(see \eqref{eq_sminormal_is_RS})}$$
obtained from its normalization $\nu \colon Y^\nu \longrightarrow Y$ and the finite set of reduced effective divisors $\{D_i = \nu^{-1}(y_i)\}_{1 \leq i \leq l}$ on $Y^\nu$. For each $C \in \mathcal{C}_{{\sim}}$, consider the reduced effective divisor $E_C = \sum_{y \in C} \nu^{-1}(y)$. Let $\Lambda$ be the indexing set obtained from $\{1, \ldots, l\}$ by removing all $i$ such that $y_i \in C$ for some $C \in \mathcal{C}_{{\sim}}$. We define $Y/{\sim}$ as the RS curve
\begin{equation}\label{eq_identification_amalgam}
Y/{\sim} \coloneqq Y^\nu \sqcup_{Y'} \left( Y/{\sim} \right)_{\rm sing}
\end{equation}
where $Y' = \sqcup_{i \in \Lambda} \nu^{-1}(y_i) \bigsqcup \sqcup_{\substack{y \in C\\ C \in \mathcal{C}_{{\sim}}}} \nu^{-1}(y)$. Since each $E_C$ and $D_j$ are reduced effective divisors having pairwise disjoint support, the curve $Y/{\sim}$ is a seminormal curve and $Y^\nu \longrightarrow Y/{\sim}$ is the normalization map. Since we have a factorization $Y^\nu \longrightarrow Y \overset{q}\longrightarrow Y/{\sim}$, the statements about the map $q$ follows.

\begin{remark}\label{rmk_factorization}
Let $X$ be a connected seminormal curve over an algebraically closed field of arbitrary characteristic. The normalization map $\nu\colon X^\nu \longrightarrow X$ factors as a composition
\begin{equation}\label{eq_factorization}
\nu \colon X^\nu = Y_t \longrightarrow Y_{t-1} \longrightarrow \cdots \longrightarrow Y_1 \longrightarrow Y_0 = X
\end{equation}
where each $Y_i$ is a seminormal curve, and $Y_i$ is obtained from $Y_{i+1}$ by identifying two distinct closed points. In other words, there are points $y_1, \, y_2 \in Y_{i+1}$ such that $Y_i = Y_{i+1}/{\sim}$ and the only non-trivial equivalence class of the equivalence relation ${\sim}$ is $\{y_1,y_2\}$.
\end{remark}

In the following, we see that a finite cover of curves descent under identification maps under a natural condition. This result will be used in ours proofs of the main results in Section~\cref{sec_Main_proj}.

\begin{lemma}\label{lem_cover_by_identification}
Let $f \colon Z \longrightarrow Y$ be a finite surjective cover of seminormal $k$-curves that is \'{e}tale over the singular points of $Y$. Let ${\sim}$ and ${\sim}'$ be two equivalence relations on the closed points of $Y$ and $Z$, respectively, such that the set of non-trivial equivalence classes $\mathcal{C}_{{\sim}}$ and $\mathcal{C}_{{\sim}'}$ are both non-empty finite sets. Additionally, suppose that the following hold.
\begin{enumerate}
\item The relations preserve $f$, i.e. $z {\sim}' z'$ implies $f(z) {\sim} f(z')$.\label{it:1}
\item For any $C \in \mathcal{C}_{{\sim}}$, the set $f^{-1}(C)$ has a partition by sets in $\mathcal{C}_{{\sim}'}$ having the same size as $C$.\label{it:2}
\end{enumerate}
Then $f$ descends to a finite surjective cover $g \colon Z/{\sim}' \, \longrightarrow \, Y/{\sim}$ that is \'{e}tale over the singular points of $Y/{\sim}$. Moreover, if $z \in Z$ is a point, and $f$ is \'{e}tale at all points in the class containing $z$, then the cover $g$ is also \'{e}tale at the image of $z$ in $Z/{\sim}'$ under the identification map $Z \longrightarrow Z/{\sim}'$.

Further assume that $f$ is $G$-Galois for a finite group $G$ and the $G$-action on $Z$ commutes with the relation ${\sim}'$, namely, for any $g \in G$ and $z \in Z$,
$$z {\sim}' z' \Rightarrow g \cdot z {\sim}' g \cdot z'.$$
Then the cover $g$ is also $G$-Galois.
\end{lemma}

\begin{proof}
Consider the normalization maps $Y^\nu \longrightarrow Y$ and $Z^\nu \longrightarrow Z$. Let $\{D_i\}_{i \in I}$ (respectively, $\{E_j\}_{j \in J}$) be the set of effective reduced divisors with pairwise disjoint support on $Y^\nu$ (respectively, on $Z^\nu$) defining $Y$ (respectively, $Z$). Then $Y/{\sim}$ is obtained from $Y^\nu$ as an RS curve together with the set of reduced effective divisors $\{D_{i'}\}_{i' \in I'} \sqcup \{D_C \coloneqq \sum_{y \in C} y\}_{C \in \mathcal{C}_{{\sim}}}$ where $I'$ is the indexing set obtained from $I$ by removing $i \in I$ such that $D_i = \nu^{-1}(y)$ for some $y \in C$, \, $C \in \mathcal{C}_{{\sim}}$. Then we have $Y/{\sim} = Y^\nu \sqcup_{Y'} \left(Y/{\sim} \right)_{\text{\rm sing}}$ (see \eqref{eq_identification_amalgam}). Similarly, we have $Z/{\sim}' = Z^\nu \sqcup_{Z'} \left( Z/{\sim}' \right)_{\text{\rm sing}}$ for a closed subscheme $Z' \subset Z^\nu$ associated to an effective reduced divisor on $Z^\nu$.

Let $f \colon Z \longrightarrow Y$ be a finite cover that is \'{e}tale over the singular points of $Y$. Then $f$ corresponds to an admissible finite cover $(f^\nu, f_{\text{ \rm sing}})$ as in Definition~\ref{def_morphism_RS}. Here $f^{\nu} \colon Z^\nu \longrightarrow Y^\nu$ is a finite cover of smooth projective $k$-curves and $f_{\text{ \rm sing}}$ is a surjection of the singular loci. Under the assumptions on $f$, the map $f_{\text{ \rm sing}}$ induces a surjection $g_{\text{ \rm sing}} \colon \left( Z/{\sim}' \right)_{\text{\rm sing}} \longrightarrow \left( Y/{\sim} \right)_{\text{\rm sing}}$. Moreover, the pair $(f^\nu, g_{\text{ \rm sing}})$ satisfies the condition for an admissible finite cover of RS curves between $Z/{\sim}'$ and $Y/{\sim}$. By the universal property of cocartesian diagrams, we obtain a finite cover $g \colon Z/{\sim}' \longrightarrow Y/{\sim}$.

By our construction, the pullback of $g$ via $Y \longrightarrow Y/{\sim}$ is the cover $f$. So the covers $g$ and $f$ agree away from the points in $C$ \, ($C \in \mathcal{C}_{{\sim}'}$). In particular, if $f$ is \'{e}tale at such a point, then so is $g$.

Now suppose that $z \in C$ for some $C \in \mathcal{C}_{{\sim}'}$. Set $D \in \mathcal{C}_{{\sim}}$ to be the class containing $f(z)$. Suppose that $f$ is \'{e}tale at all points in the class containing $z$. So we have
\begin{equation}\label{eq_1}
\mathcal{O}_{Z,u} \cong \mathcal{O}_{Y,f(u)}
\end{equation}
for all $u \in C$. Suppose that $z' \in Z/{\sim}'$ and $y' \in Y/{\sim}$ be the images of $z$ and $f(z)$ under the respective identification maps. Then we have
\begin{eqnarray*}
\mathcal{O}_{Z/{\sim}',z'} = k + J\left(\bigcap_{\substack{c \in C \\ u \in \nu^{-1}(c)}} \mathcal{O}_{Z^\nu, u}\right) \hspace{2 em} \text{and}\\
\mathcal{O}_{Y/{\sim},y'} = k + J\left(\bigcap_{\substack{d \in D \\ v \in \nu^{-1}(d)}} \mathcal{O}_{Y^\nu, v}\right)
\end{eqnarray*}
As $|C| = |D|$ by our assumption, the isomorphisms in \eqref{eq_1} induce an isomorphism $\mathcal{O}_{Z/{\sim}', z'} \cong \mathcal{O}_{Y/{\sim}, y'}$. So the map $g$ is \'{e}tale at $z$.

Finally, if $f$ is $G$-Galois and the $G$-action preserves the relations, the map $g$ obtained using the universal property of cocartesian diagrams is also $G$-Galois.
\end{proof}

\section{Case of Projective Curves}\label{sec_proj}
\subsection{Main Results in Projective Case}\label{sec_Main_proj}
Throughout this section, let $k$ be an algebraically closed field of a prime characteristic $p$. We fix the following notation.

\begin{notation}\label{not_set_up}
Let $X$ be a connected projective $k$-curve (reduced, but not necessarily irreducible). Suppose that $X_{\text{Sing}} = \{x_1, \ldots, x_l\}$, \, $l \geq 1$, is the singular locus of $X$. Let $\nu \colon X^\nu = \sqcup_{1\leq i \leq n} C_i \longrightarrow X$ be the normalization map where $C_i$'s are the irreducible components of $X^\nu$. The finite birational morphism $\nu$ factors as
$$X^\nu \longrightarrow X^{+} \longrightarrow X$$
where $X^{+}$ is the seminormalization (Definition~\ref{def_seminormal_curves}) of $X$.
\end{notation}

Our goal is to relate the \'{e}tale fundamental group $\pi_1(X)$ of the projective curve $X$ with $\pi_1(X^\nu) = \pi_1(C_1) \ast \cdots \ast \pi_1(C_n)$. We first observe the following result.

\begin{lemma}\label{lem_same_pi_for_semi-normal}
Under the above notation, $X^{+}$ is a connected projective $k$-curve and we have
$$\pi_1(X) \cong \pi_1(X^{+}).$$
\end{lemma}

\begin{proof}
Note that $X^+ \longrightarrow X$ is a subintegral map, i.e. for each open affine set $U = \text{\rm Spec}(A) \subset X$ with its preimage $V = \text{\rm Spec}(B) \subset X^+$, the map $V \longrightarrow U$ is a homeomorphism of the underlying topological spaces that induces isomorphisms of residue fields (\cite[Theorem~2.9(1), page 448]{Vitulli}). In particular, the map $X^+ \longrightarrow X$ is a radical morphism. As $X$ is connected, so is $X^+$. Since the normal curve $X^\nu$ is projective, $X^+$ is also projective.

The normalization map $\nu$ is a finite surjective morphism. So $X^+ \longrightarrow X$ is a surjective integral and radical morphism. By \cite[\href{https://stacks.math.columbia.edu/tag/04DF}{Lemma 04DF}]{SP}, the map $X^+ \longrightarrow X$ is a universal homeomorphism. Now the result follows from \cite[Exp. IV, Theorem~18.1.2, page 110]{EGA}.
\end{proof}

In view of the above lemma, we further assume that \emph{$X$ is seminormal.}

Now consider a complete discrete valuation ring $R$ with residue field $k$ and fraction field $K$ of characteristic $0$. Let $\bar{K}$ be an algebraic closure of $K$. By Proposition~\ref{prop_lift_seminormal}, there us a connected $R$-curve $X_R$, the map $X_R \longrightarrow \text{\rm Spec}(R)$ is flat, proper with geometrically connected fibres. We obtain the specialization map (\cite[Exp.10]{SGA1} or \cite[\href{https://stacks.math.columbia.edu/tag/0BUP}{Section 0BUP}]{SP})
$$\text{\rm sp} \colon \pi_1(X_{\bar{K}}) \longrightarrow \pi_1(X)$$
of Grothendieck that is a continuous homomorphism. Since $X$ is a reduced curve, the map $\text{\rm sp}$ is surjective by \cite[\href{https://stacks.math.columbia.edu/tag/0C0P}{Lemma 0C0P}]{SP}. Without lose of generality, we may assume (see the argument in \cite[Remark~1.7]{Obus}) that $\bar{K}=\mathbb{C}$. Then the above surjection becomes
\begin{equation}\label{eq_surj}
\text{\rm sp} \colon \pi_1(X_{\mathbb{C}}) \twoheadrightarrow \pi_1(X).
\end{equation}
Further, by Proposition~\ref{prop_lift_seminormal}, $X_{\mathbb{C}}$ is obtained from its normalization $X_{\mathbb{C}}^\nu = \sqcup_{1\leq i \leq n} D_i$ via identifying points the same way $X$ is obtained from $X^\nu$ (here $D_i$'s are the irreducible components of $X^\nu$), and $g(C_i) = g(D_i)$ for each $i$. By the general version of the Riemann Existence Theorem (\cite[Exp XII; Corollary~5.2]{SGA1}), $\pi_1(X_{\mathbb{C}})$ is the profinite completion of the topological fundamental group $\pi_1^{\text{\rm top}}(X_{\mathbb{C}}(\mathbb{C}))$.

We have a complete description of the group $\pi_1(X_{\mathbb{C}})$ as follows. By Remark~\ref{rmk_factorization}, the normalization map $\nu_{\mathbb{C}} \colon X_{\mathbb{C}}^\nu \longrightarrow X_{\mathbb{C}}$ factors as a composition
$$X_{\mathbb{C}}^\nu = Y_t \longrightarrow Y_{t-1} \longrightarrow \cdots \longrightarrow Y_1 \longrightarrow Y_0 = X_{\mathbb{C}}$$
where each $Y_i$ is a seminormal projective $\mathbb{C}$-curve and $Y_i$ is obtained from $Y_{i+1}$ by identifying two distinct closed points $y_{i+1,1}$ and $y_{i+1,2}$. The underlying topological space $Y_{i}(\mathbb{C})$ is a deformation retract of the space $U_{i+1}$ obtained from $Y_{i+1}(\mathbb{C})$ by joining $y_{i+1,1}$ to $y_{i+1,2}$ via a contractible arc (i.e. a `handle'). Then $\pi_1^{\text{\rm top}}(Y_i) = \pi_1^{\text{\rm top}}(U_{i+1})$. Using Van Kampen's Theorem, $\pi_1^{\text{\rm top}}(U_{i+1}) = \pi_1^{\text{\rm top}}(Y_{i+1}(\mathbb{C})) \ast \mathbb{Z}$ if the points $y_{i+1,1}$ and $y_{i+1,2}$ lie in a connected component, and $\pi_1^{\text{\rm top}}(U_{i+1}) = \pi_1^{\text{\rm top}}(Y_{i+1}(\mathbb{C}))$ otherwise. We conclude that $\pi_1(X_{\mathbb{C}})$ is a free product of finitely generated profinite groups
$$\pi_1(X_{\mathbb{C}}) = \pi_1(D_1) \ast \cdots \ast \pi_1(D_n) \ast \widehat{F_{\delta}} = \pi_1(X^\nu_{\mathbb{C}}) \ast \widehat{F_{\delta}},$$
and $F_{\delta}$ is a free finitely generated group on $\delta \coloneqq 1-n+\sum_{1 \leq i \leq l} \left( |\nu^{-1}(x_i)| -1 \right)$ generators, recursively defined as follows. Set $F_t$ to be the trivial group. If the two points in $Y_{i+1}$ getting identified under $Y_{i+1} \longrightarrow Y_i$ lie in the same connected component of $Y_{i+1}$, set $F_{i} \coloneqq F_{i+1} \ast \mathbb{Z}$. Otherwise, set $F_i \coloneqq F_{i+1}$. If $g_i = g(D_i)$ denote the genus of the compact Riemann surface $D_i$, $1 \leq i \leq n$, it is well known that the \'{e}tale fundamental group $\pi_1(D_i)$ is a the profinite completion of a finitely generated group on $2g_i$ generators $a_1, b_1, \ldots, a_{g_i}, b_{g_i}$ subject to the only relation $\prod_{1 \leq j \leq g_i} \left[ a_j, b_j \right] =1$.

As a consequence of the surjective specialization homomorphism $\text{\rm sp}$ in~\eqref{eq_surj} and the description of the group $\pi_1(X_{\mathbb{C}})$ as above, $\pi_1(X)$ is a finitely generated group of rank at most $2 \sum_{1 \leq i \leq n} g_i + \delta$\, (recall that $g(C_i) = g_i = g(D_i)$), and
$$\delta = 1-n+\sum_{1 \leq i \leq l} \left( |\nu^{-1}(x_i)| -1 \right).$$
In particular, $\pi_1(X)$ is completely determined by its finite group quotients (as a topological group; \cite[Theorem~1.1]{FG}).

We further see that there is a surjective homomorphism
$$\pi_1(C_1) \ast \cdots \ast \pi_1(C_n) \ast \widehat{F_{\delta}} \twoheadrightarrow \pi_1(X)$$
of finitely generated profinite groups. Recall that we obtained a surjective specialization map $\text{\rm sp} \colon \pi_1(X_{\mathbb{C}}) \twoheadrightarrow \pi_1(X)$ in~\eqref{eq_surj}. In terms of covers, the existence of the specialization map is equivalent to: the functor $\text{\rm FEt}_{X_R} \longrightarrow \text{\rm FEt}_X$ taking covers to their special fibres is an equivalence of categories (with the quasi inverse $\theta$, say). Then we obtain a functor
$$\lambda \colon \text{\rm FEt}_X \overset{\theta}{\underset{\sim} \longrightarrow} \text{\rm FEt}_{X_R} \longrightarrow \text{\rm FEt}_{X_{\mathbb{C}}}.$$
The surjectivity of $\text{\rm sp}$ is equivalent to: $\lambda$ maps connected finite \'{e}tale covers of $X$ to connected finite \'{e}tale covers of $X_{\mathbb{C}}$.

Now let $f \colon Z \longrightarrow X$ be a $G$-Galois \'{e}tale cover of connected $k$-curves. Using the surjection $\text{\rm sp}$ in~\eqref{eq_surj}, $G$ is generated by its subgroups $G_1, \ldots, G_n$ and $H$ with $F_\delta \twoheadrightarrow H$ and $\pi_1(D_i) \twoheadrightarrow G_i$ (as before, $D_i$'s are the irreducible components of $X_{\mathbb{C}}^\nu$). So the $G$-Galois cover $\nu^*(\lambda(f))$ is a disjoint union of $G_i$-Galois \'{e}tale connected covers of $D_i$'s. We have the following $2$-commutative diagram of Galois categories.
\begin{center}
\begin{equation*}\label{diag_2_comm}
\begin{tikzcd}
\lambda^\nu \colon \text{\rm FEt}_{X^\nu} \arrow{r}{\sim} & \text{\rm FEt}_{X_R^\nu} \arrow{r} & \text{\rm FEt}_{X^\nu_{\mathbb{C}}} \\
\lambda \colon \text{\rm FEt}_X \arrow{r}{\sim} \arrow{u}{\nu^*} & \text{\rm FEt}_{X_R} \arrow{r} \arrow{u}{\nu^*} & \text{\rm FEt}_{X_{\mathbb{C}}} \arrow{u}{\nu^*}
\end{tikzcd}
\end{equation*}
\end{center}
As $\lambda^{\nu}(\nu^*(f)) = \nu^*(\lambda(f))$, the $G$-Galois cover $\nu^*(f)$ is also a disjoint union of $G_i$-Galois \'{e}tale connected covers of $C_i$'s. This shows that each $G_i$ is in fact a finite quotient of $\pi_1(C_i)$. As the profinite groups in question are finitely generated, we obtain a surjective homomorphism
\begin{equation}\label{eq_one_way}
\pi_1(C_1) \ast \cdots \ast \pi_1(C_n) \ast \widehat{F_{\delta}} \twoheadrightarrow \pi_1(X).
\end{equation}

We have the following structure theorem for the fundamental group $\pi_1(X)$ in terms of its normalizer and singular points.

\begin{theorem}\label{thm_main_projective}
Let $k$ be an algebraically closed field of arbitrary characteristic. Let $X$ be a connected projective $k$-curve. Let $X_{\rm sing} = \{x_1, \ldots, x_l\}$. Let $\nu \colon X^\nu = \sqcup_{1\leq i \leq n} C_i \longrightarrow X$ be the normalization map where $C_i$'s are the irreducible components of $X^\nu$. Then the \'{e}tale fundamental group $\pi_1(X)$ is the following profinite group.
$$\pi_1(X) = \pi_1(C_1) \ast \cdots \ast \pi_1(C_n) \ast \widehat{F_\delta}$$
where $\widehat{F_\delta}$ is the profinite completion of a free group $F_\delta$ on
$$\delta \coloneqq 1-n+\sum_{1 \leq i \leq l} \left( |\nu^{-1}(x_i)| -1 \right)$$
generators. In particular, $\pi_1(X)$ is a finitely generated profinite group on $\sum_{1 \leq i \leq n} r_i + \delta$ generators, where $r_i$ is the rank of $\pi_1(C_i)$ \, ($r_i \leq 2 g(C_i)$).
\end{theorem}

\begin{proof}
We assume that $\text{\rm char}(k) > p$ and $X_{\text{\rm sing}} \neq\emptyset$ since the statement is known otherwise. In view of Lemma~\ref{lem_same_pi_for_semi-normal}, we may assume that $X$ is a seminormal curve. By our previous discussion, $\pi_1(X)$ is a finitely generated group of rank at most $2 \sum_{1 \leq i \leq n} g(C_i) + \delta$, and hence it is completely determined as a topological group by its finite quotients.

The normalization map $\nu$ admits a factorization (see \eqref{eq_factorization})
$$\nu \colon X^\nu = Y_t \longrightarrow Y_{t-1} \longrightarrow \cdots \longrightarrow Y_1 \longrightarrow Y_0 = X$$
where each $Y_i$ is a seminormal curve, and $Y_i$ is obtained from $Y_{i+1}$ by identifying two distinct closed points. By inductively applying the next two propositions, we show that any finite quotient of $\pi_1(C_1) \ast \cdots \ast \pi_1(C_n) \ast \widehat{F_{\delta}}$ is also a finite quotient of $\pi_1(X)$, and hence determine a surjection
$$\pi_1(X) \twoheadrightarrow \pi_1(C_1) \ast \cdots \ast \pi_1(C_n) \ast \widehat{F_{\delta}}.$$
We already have a surjective homomorphism $\pi_1(C_1) \ast \cdots \ast \pi_1(C_n) \ast \widehat{F_{\delta}} \twoheadrightarrow \pi_1(X)$ in~\eqref{eq_one_way}, and this concludes the proof of the theorem.
\end{proof}

\begin{proposition}\label{prop_main_1}
Let $Y$ be a seminormal projective $k$-curve, $y_1, \, y_2 \in Y$ be two closed points. Suppose that $X$ is the seminormal curve obtained from $Y$ by identifying $y_1$ with $y_2$ (see \eqref{eq_identify}). Assume that $Y$ is connected. Let $\Gamma$ be a finite group generated by a subgroup $G$ and an element $\gamma \in \Gamma$. Let $f \colon Z \longrightarrow Y$ be a $G$-Galois cover of connected projective $k$-curves, \'{e}tale over $Y_{\text{\rm sing}} \cup \{y_1, y_2\}$. Then there exists a $\Gamma$-Galois cover $g \colon W \longrightarrow X$ of connected projective $k$-curves, \'{e}tale over $X_{\text{\rm sing}}$.
\end{proposition}

\begin{proposition}\label{prop_main_2}
Let $Y$ be a seminormal projective $k$-curve, $y_1, \, y_2 \in Y$ be two closed points. Suppose that $X$ is the seminormal curve obtained from $Y$ by identifying $y_1$ with $y_2$ (see \eqref{eq_identify}). Assume that $Y$ is a union of two distinct connected components $Y_1$ and $Y_2$ containing the points $y_1$ and $y_2$, respectively. Let $G$ be a finite group generated by two of its subgroups $G_1$ and $G_2$. Suppose that for $a = 1, \, 2$, \, $f_a \colon Z_a \longrightarrow Y_a$ be a $G_i$-Galois cover of connected projective $k$-curves, \'{e}tale over $\left( Y_a \right)_{\text{\rm sing}} \cup \{y_a\}$. Then there exists a $G$-Galois cover $g \colon W \longrightarrow X$ of connected projective $k$-curves, \'{e}tale over $X_{\text{\rm sing}}$.
\end{proposition}

Before proving the above propositions, we make some observation on the Galois action on the induced covers. Let $X$ be any projective connected $k$-curve, and let $f \colon Z \longrightarrow X$ be a $G$-Galois cover of connected $k$-curves that is \'{e}tale over the singular locus $X_{\text{\rm sing}}$ of $X$. This implies that the image of the singular points of $Z$ are mapped to singular points of $X$, and the fibre of $f$ over any singular point of $X$ consists only of singular points of $Z$. Fix a closed point $x \in X$ in the \'{e}tale locus of $f$. Then $G$ acts simply transitively on the finite set of points $f^{-1}(x)$ in $Z$. In other words, the closed subscheme $f^{-1}(x)$ of $Z$ is a principal homogeneous space for the finite constant $k$-group scheme $G$ over $\text{\rm Spec}(k)$.

We have the following group theoretic observation. For any finite group $G$ and a finite set $S$, there is a bijective correspondence
\begin{equation}\label{eq_rep}
\begin{Bmatrix} \text{ A simply transitive }G\text{-action on }S \\ \text{ together with an element } s_0 \in S \end{Bmatrix} \overset{\sim} \longrightarrow \begin{Bmatrix} \text{ A bijection } \alpha \colon S \longrightarrow G \end{Bmatrix}
\end{equation}
defined by $(\Lambda \colon G \times S \longrightarrow S) \, \mapsto \, \alpha(s) = g$ where $g \in G$ is the unique element $\Lambda(g,s)=s_0$ obtained by the transitivity of $\Lambda$ (Here $\alpha^{-1}(g) = \Lambda(g,s_0)$). The inverse map of the above correspondence is given by $\alpha \mapsto \left( \Lambda, \alpha^{-1}(e)\right)$, \, $e$ is the identity element of $G$, \, $\Lambda(g, s) = \alpha^{-1} \left( g \cdot \alpha(s) \right)$ where $g \cdot \alpha(s)$ is obtained by the regular action of $G$ on itself.

Now suppose that $\Gamma$ is a finite group generated by its subgroups $G$ and $H$. Consider the induced (disconnected) $\Gamma$-Galois cover $f^{\text{ \rm ind}} \colon \text{\rm Ind}_G^{\Gamma} Z \longrightarrow X$. Then $f^{\text{ \rm ind}}$ is also \'{e}tale over $X_{\text{\rm sing}}$. Since $f$ is \'{e}tale over the point $x \in X$, so is the cover $f^{\text{ \rm ind}}$. We study the $\Gamma$-action on the fibre $\left( f^{\textbf{ \rm ind}} \right)^{-1}(x)$ in detail. Fix a point $z_0 \in f^{-1}(x)$.

Fix a set of coset representatives $C(\Gamma/G) = \{e=h_0, h_1, \ldots, h_{N-1}\}$ of the left coset $\Gamma/G$ where $N = [\Gamma \colon G]$ and $h_i \in H$. Then $\text{\rm Ind}_G^{\Gamma} Z = \sqcup_{0 \leq i \leq N-1} Z_i$ is a disjoint union where each $Z_i$ is a copy of $Z$ indexed by $h_i \in C(\Gamma/G)$. Let $z_0^{(i)} \in Z_i$ be the unique point over $z_0 \in Z$ under the above identification $Z_i \cong Z$. As $f$ is \'{e}tale over $x$, for each $0 \leq i \leq N-1$, the simply transitive $G$-action on $f^{-1}(x)$ and the point $z_0^{(i)}$ determines a unique bijection by~\eqref{eq_rep}:
$$\alpha_i \colon \left(f^{\text{ \rm ind}}\right)^{-1}(x) \cap Z_i \overset{{\sim}} \longrightarrow G.$$
These define a bijection
\begin{equation}\label{eq_bij}
\alpha \colon \left(f^{\text{ \rm ind}}\right)^{-1}(x) \longrightarrow \Gamma
\end{equation}
as follows. For $z \in \left(f^{\text{ \rm ind}}\right)^{-1}(x)$, there is a unique $0 \leq i \leq N-1$ such that $z \in \left(f^{\text{ \rm ind}}\right)^{-1}(x) \cap Z_i$. Set $\alpha(z) \coloneqq h_{i} \alpha_i(z)$. It has the inverse map given by: $\gamma \in \Gamma$ can be written uniquely as $\gamma = h_i g$ for $i$ and $g \in G$; set $\alpha^{-1}(\gamma) \coloneqq \alpha_i^{-1}(g) = g \cdot z_0^{(i)}$. By \eqref{eq_rep}, this bijection corresponds to the point $\alpha^{-1}(e) = z_0^{(0)} \in Z_0$ and the $\Gamma$-action $\Lambda$ on $\left( f^{\text{ \rm ind}} \right)^{-1}(x)$ given by
$$\Lambda(\gamma , z) = \alpha^{-1} \left( \gamma \cdot \alpha(z) \right)$$
where $\gamma \cdot \alpha(z)$ is obtained from the regular $G$-action on itself. More precisely, let $\gamma = h_i g, \,z \in Z_j$. Then $h_i g \alpha_j(z) = h_{i'}g'$ for some uniquely determined $i'$ and $g' \in G$. We have $\Lambda(\gamma , z) = g' \cdot z_0^{(i')}$ determined by the $G$-action on $f^{-1}(x)$.

\begin{proof}[{\underline{Proof of Proposition~\ref{prop_main_1}}}]
Let $x$ be the image of $y_1$ (and of $y_2$) under the identification map $Y \longrightarrow X$. By our hypothesis, $f$ is possibly ramified only over a finitely many smooth points of $X$. For any closed point $z \in Z$ over which $f$ is \'{e}tale, we have $\mathcal{O}_{Z,z} \cong \mathcal{O}_{Y,f(z)}$. If $f$ is ramified at a point $z \in Z$, then $f(z)$ is a smooth point of $X$; so the local ring extension $\mathcal{O}_{X,x} \subset \mathcal{O}_{Z,z}$ is necessarily an extension of discrete valuation domains. Thus $Z$ is also a seminormal curve (see Definition~\ref{def_seminormal_curves}).

Let $N = [\Gamma \colon G]$. Then $C(\Gamma/G) = \{e = \gamma^0, \gamma, \ldots, \gamma^{N-1}\}$ is a set of coset representatives for the left coset $\Gamma/G$. Consider the (possibly disconnected) $\Gamma$-Galois \'{e}tale cover
$$f^{\text{ \rm ind}} \colon \text{Ind}_G^{\Gamma} Z = \sqcup_{0 \leq i \leq N-1} Z_i \longrightarrow Y$$
where each $Z_i$ is a copy of $Z$ indexed by $\gamma^i$. Let $a \in \{1,2\}$. Fix a point $z_a \in f^{-1}(y_a)$, and let $z_a^{(i)} \in Z_i$ denote its image under the isomorphism $Z_i \cong Z$. We have the corresponding bijections ($0 \leq i \leq N-1$)
$$\alpha_{a,i} \colon \left( f^{\text{ \rm ind}} \right)^{-1}(y_a) \cap Z_i \overset{{\sim}} \longrightarrow G.$$
Recall from \eqref{eq_bij} that the above bijections define a bijection $\alpha_a \colon \left( f^{\text{ \rm ind}} \right)^{-1}(y_a) \longrightarrow \Gamma$ that determine the $\Gamma$-action on $\left( f^{\text{ \rm ind}} \right)^{-1}(y_a)$.

We define an equivalence relation ${\sim}$ on $\text{\rm Ind}_G^{\Gamma} Z$ whose non-trivial classes are obtained using the following rule.
\begin{itemize}
\item A point $z \in \left( f^{\text{ \rm ind}} \right)^{-1}(y_1) \cap Z_i$ is identified with a point $z' \in \left( f^{\text{ \rm ind}} \right)^{-1}(y_2) \cap Z_{i+1 \pmod{N}}$ if and only if $\alpha_{1,i}(z) = \alpha_{2,i+1 \pmod{N}}(z')$.
\end{itemize}
Let $W$ be the seminormal curve $\left( \text{\rm Ind}_G^{\Gamma} Z \right)/{\sim}$ (see \eqref{eq_identify}); the finite birational morphism $\beta \colon \text{\rm Ind}_G^{\Gamma} Z \longrightarrow W$ is an isomorphism away from the set of points $S \coloneqq \left( f^{\text{ \rm ind}} \right)^{-1}\left( \{y_1, y_2\} \right)$. By our construction, the equivalence relations satisfy the hypothesis of~\eqref{it:1} and~\eqref{it:2} of Lemma~\ref{lem_cover_by_identification}. Moreover, $f$ is \'{e}tale over both $y_1$ and $y_2$. As $\alpha_{a,i}(g \cdot z_a^{(i)}) = g$ for any $g \in G, \, 0 \leq i \leq N-1$, we have
$$z_1^{(i)} {\sim} z_2^{(i+1 \pmod{N})} \text{ and for any } g \in G, \, g \cdot z_1^{(i)} {\sim} z_2^{(i+1 \pmod{N})}.$$
Thus $f$ induces a $\Gamma$-Galois cover $g \colon W \longrightarrow X$ that is \'{e}tale over $X_{\text{\rm sing}}$. Since for each $0 \leq i \leq N-1$, the set $\{z \in \text{\rm Ind}_G^{\Gamma} Z \, | \,  z \text{ lies in a non-trivial equivalence class of } {\sim}\} \cap Z_i$ is non-empty and $Z$ is connected, $W$ is a connected seminormal curve.
\end{proof}

\begin{proof}[{\underline{Proof of Proposition~\ref{prop_main_2}}}]
Let $x \in X$ be the image of $y_1$ ( and of $y_2$) under the identification map $Y \longrightarrow X$. It also follows that $Z_1$ and $Z_2$ are seminormal curves. Let $a \in \{1, 2\}$. Since $G = \langle \, G_1, G_2 \, \rangle$, we choose a set of coset representatives $C(G/G_a) = \{e=g_{3-a,0}, g_{3-a,1}, \ldots, g_{3-a,n_a-1}\}$ for the left coset $G/G_a$ where $n_a = [G \colon G_a]$ and each $g_{3-a,i_a} \in G_{3-a}$. Consider the induced (possibly disconnected) $\Gamma$-Galois cover
$$f_a^{\text{ \rm ind}} \colon \text{\rm Ind}_{G_a}^G Z_a = \sqcup_{0 \leq i_a \leq n_a-1} Z_{a,i_a} \longrightarrow Y_a,$$
\'{e}tale over $\left( Y_a \right)_{\text{\rm sing}} \cup \{y_a\}$; each $Z_{a,i_a}$ is a copy of $Z_a$ indexed by $g_{3-a,i_a} \in C(G/G_a)$. Fix a point $z_a \in f_a^{-1}(y_a)$. We obtain bijections (see \eqref{eq_bij})
$$\alpha_a \colon \left( f_a^{\text{ \rm ind}} \right)^{-1}(y_a) \overset{{\sim}} \longrightarrow G$$
that determines the $G$-action on the set $\left( f_a^{\text{ \rm ind}} \right)^{-1}(y_a)$.

We have the following (possibly disconnected) $G$-Galois cover
$$\phi \colon f_1^{\text{ \rm ind}} \sqcup f_2^{\text{ \rm ind}} \colon Z' \coloneqq \text{\rm Ind}_{G_1}^G Z_1 \sqcup \text{\rm Ind}_{G_2}^G Z_2 \longrightarrow Y_1 \sqcup Y_2$$
that is \'{e}tale over $Y_{\text{\rm sing}} \cup \{y_1, y_2\}$. Consider the equivalence relation ${\sim}$ on the closed points of $Z'$ where the non-trivial classes are given as follows.
\begin{itemize}
\item A point $z_1 \in \left( f_1^{\text{ \rm ind}} \right)^{-1}(y_1)$ is identified with a point $z_2 \in \left( f_2^{\text{ \rm ind}} \right)^{-1}(y_2)$ if and only if $\alpha_1(z_1) = \alpha_2(z_2)$.
\end{itemize}
In other words, if $g \in G$, then $g$ can be uniquely written as $g = g_{1,i} g_2 = g_{2,j} g_1, \, g_1 \in G_1, \, g_2 \in G_2, \, 0 \leq i \leq n_2, \, 0 \leq j \leq n_1$. We identify the point $\alpha_{1,j}^{-1}(g_1)$ on $Z_{1,j}$ with the point $\alpha_{2,i}^{-1}(g_2)$ on $Z_{2,i}$. Let $\beta \colon Z' \longrightarrow W \coloneqq Z'/{\sim}$ be the identification map. It is easy to check that the conditions of Lemma~\ref{lem_cover_by_identification} are satisfied, and we obtain a $G$-Galois cover $W \longrightarrow X$, \'{e}tale over $X_{\text{\rm sing}}$.

Now we show that $W$ is connected. It is enough to show that for any $0 \leq i \leq n_1-1$ and any $0 \leq j \leq n_2-1$, there exist points $z_1 \in Z_{1,j}, \, z_2 \in Z_{2,i}$ that are identified via ${\sim}$. Consider any point $z_1 = \alpha_{1,j}^{-1}(g_1)$ on $Z_{1,j}$. Then $g \coloneqq g_{2,j} g_1$ can be written uniquely as $g = g_{1,i} g_2$, and $z_1$ is identified with $z_2 \coloneqq \alpha_{2,i}^{-1}(g_2)$ on $Z_{2,i}$.
\end{proof}

\begin{remark}\label{rmk_preserving_inertia}
The above proofs also show the following. Let $y \in Y$ be any smooth point, $y \neq y_1, \, y_2$. Let $x$ denote the smooth point in $X$ obtained as the image of $y$ under the identification map $Y \longrightarrow X = Y/{y_1 {\sim} y_2}$. If $I_y$ is an inertia group over $y$ for the cover $f$ in Proposition~\ref{prop_main_1} (or for either of the covers $f_1$ and $f_2$ in Proposition~\ref{prop_main_2}), then $I_y$ is also an inertia group above $x$ for the cover $W \longrightarrow X$ (in either of the propositions). So the above constructions do not change the ramification behavior above any point.
\end{remark}

\begin{remark}\label{rmk_categorical}
In Theorem~\ref{thm_main_projective}, we have
$$\pi_1(X) = \pi_1(C_1) \ast \cdots \ast \pi_1(C_n) \ast \widehat{F_\delta} = \pi_1(X^\nu) \ast \widehat{F_\delta}.$$
In particular, we obtain the injective homomorphism $\pi_1(X^\nu) \hookrightarrow \pi_1(X)$. The corresponding exact functor of the Galois categories is the pullback of finite \'{e}tale covers of $X$ under the normalization map. Similarly, each surjection $\pi_1(X) \longrightarrow \pi_i(C_i)$ corresponds to the exact fully faithful (\cite[\href{https://stacks.math.columbia.edu/tag/0BN6}{Lemma 0BN6}]{SP}) functor $\text{\rm FEt}_{C_i} \longrightarrow \text{\rm FEt}_{X}$. These are precisely the maps obtained (for connected objects in the categories) in Proposition~\ref{prop_main_1} and Proposition~\ref{prop_main_2}. So the above propositions describe a canonical way (up to the choices of coset representatives and of one point in each fibre of covers over the points getting identified) to obtain all the Galois covers of any singular curve $X$ starting from the Galois covers of its normalization.
\end{remark}

\subsection{Necessary Conditions on Galois groups}\label{sec_nec}
We retain the notation from our previous section. As application of our main theorem~\ref{thm_main_projective}, we have the following necessary conditions for any finite quotient of $\pi_1(X)$. These are known when $X$ is smooth (see \cite{Stevenson} or \cite[Section 7]{survey_paper}).

\begin{corollary}\label{cor_nec}
Let $X$ be a connected singular curve over an algebraically closed field $k$ of characteristic $p > 0$. Let $\delta$ and $g_i = g(C_i)$ ($1 \leq i \leq n$) be as in Theorem~\ref{thm_main_projective}. Let $X^+$ be the seminormalization of $X$.
\begin{enumerate}
\item The specialization map $\text{\rm sp} \colon \pi_1(X^+_{\mathbb{C}}) \twoheadrightarrow \pi_1(X)$ obtained by composing the surjection~\eqref{eq_surj} with the isomorphism $\pi_1(X^+) \cong \pi_1(X)$ of Lemma~\ref{lem_same_pi_for_semi-normal} induces an isomorphism of the maximal $p'$-quotients
$$\pi_1(X^+_{\mathbb{C}})^{(p')} \twoheadrightarrow \pi_1(X)^{(p')}.$$\label{item-1}
\item The maximal pro-$p$ quotient $\pi_1(X)^{(p)}$ is a free profinite group on $\sum_{1 \leq i \leq n} s_i + \delta$ generators; here $s_i$ denote the $p$-rank of $C_i$.\label{item-2}
\end{enumerate}

Further, suppose that a finite group $G$ occurs as the Galois group of a Galois \'{e}tale connected cover of $X$. Let $I_G$ denote the augmented ideal
$$\left\{\sum_{g \in G} a_g g \, | \, \sum_g a_g = 0 \right\}$$
of the group ring $k[G]$. Let $t_G$ be the minimal number of generators of $I_G$. Let $\sigma(G)$ be the $p$-rank of the abelianization $G_{\text{\rm ab}}$ of $G$. Then $G$ satisfies the following properties.
\begin{enumerate}[resume]
\item (Nakajima's Condition for singular curves)\footnote{For integral semistable curves, this can be proved using the original ideas of \cite{Nakajima}; in fact, Theorem~3 and Theorem~4 of loc. cit. holds in this set up. This insight is missing for general curves as the Riemann Roch Theorem and related notions are known when the curve is Gorenstein.} $t_G \leq \sum_{1 \leq i \leq n} g_i + \delta$.\label{item-3}
\item (Hasse-Witt for singular curves) $\sigma(G) \leq \sum_{1 \leq i \leq n} g_i+\delta$.\label{item-4}
\end{enumerate}
\end{corollary}

\begin{proof}
We first note that if $\Gamma$ is a free product of finitely generated profinite groups $\Gamma_1$ and $\Gamma_2$, then there is an isomorphism $\Gamma^{(p')} \cong \left( \Gamma_1^{(p')} \ast \Gamma_2^{(p')} \right)^{(p')}$ of maximal $p'$-quotients and an isomorphism $\Gamma^{(p)} \cong \left( \Gamma_1^{(p)} \ast \Gamma_2^{(p)} \right)^{(p)}$ of maximal pro-$p$ quotients. If $G$ is a finite quotient of $\Gamma_1^{(p')} \ast \Gamma_2^{(p')}$ of order prime-to-$p$, then $G$ is also a quotient of $\Gamma$ by the normal subgroup generated by the kernels of the maps $\Gamma \twoheadrightarrow \Gamma_i$. Conversely, suppose that $\Gamma^{(p')} \twoheadrightarrow G$ be a finite quotient. Then $G$ admits generators which are images of generators of $\Gamma$ and satisfies the relations in $\Gamma$. Consider the subgroups $G_1$ and $G_2$ of $G$ that are generated by those elements which are images of generators of $\Gamma_1$ and $\Gamma_2$, respectively. These generators necessarily satisfy the relations in respective $\Gamma_i$'s. So $G_i$ is a finite quotient of $\Gamma_i$ and is of order prime-to-$p$. Since $\Gamma_i$'s are finitely generated, we see the first isomorphism of maximal $p'$-quotients. The other isomorphism follows similarly by considering finite quotients which are $p$-groups.

In view of the above, \eqref{item-1} follows from the isomorphism induced by the specialization map in smooth case (\cite[Exp XIII, Corollary~2.12, page 392]{SGA1}). Similarly, since $\pi_1(Y)^{(p)}$ for a smooth projective $k$-curve $Y$ is a free group on $s_Y$ ($= p\text{\rm -rank of } Y$) generators (\cite{Shafarevitch}), \eqref{item-2} follows.

Now we prove \eqref{item-3}. By Theorem~\ref{thm_main_projective}, $G$ is generated by its subgroups $H_1, \ldots, H_n$ and $H$ (which may be trivial) such that $H_i$ a finite quotient of $\pi_1(C_i)$ and $H$ is a finite group on $\leq \delta$ generators. Then $t_G \leq \sum_i t_{G_i} + t_H$. By \cite[Corollary~1.4]{Stevenson}, we have $t_{H_i} \leq g_i$. We also have $t_H \leq \delta$ by \cite[Proposition~2.1(i)]{Stevenson}. Thus $t_G \leq \sum_i g_i +\delta$.

For \eqref{item-4}, suppose that $f \colon Z \longrightarrow X$ be a $G$-Galois \'{e}tale cover of connected $k$-curves. Let $A$ denote the maximal elementary abelian $p$-quotient of $G$. Then $f$ factors through a connected $A$-Galois cover $Y \longrightarrow X$. Then by \eqref{item-2}, the number of generators for $A$ ($= \sigma(G)$ by definition) is bounded above by $\sum_{1 \leq i \leq n} s_i + \delta \leq \sum_{1 \leq i \leq n} g_i + \delta$.
\end{proof}

\section{Case of Affine Curves}\label{sec_affine}
\subsection{Main Results for affine case}\label{sec_Main_Affine}
In this section, we study the Galois covers of an affine $k$-curve $U$. As before, $k$ is an algebraically closed field of characteristic $p > 0$. We also assume that $U$ is irreducible.

Even when $U$ is smooth, the \'{e}tale fundamental group $\pi_1(U)$ is not finitely generated; so understanding its finite quotients does not fully determine the structure of the profinite group $\pi_1(U)$. Nevertheless, it is important to understand the finite continuous quotients. Abhyankar's Conjecture on Affine Curves (now a theorem) states that for a smooth irreducible affine $k$-curve $U$ with its smooth projective completion $U \subset X$ and $r = |X - U|$, a group $G$ is a finite continuous quotient of $\pi_1(U)$ if and only if $G/p(G)$ has a generating set of at most $2 g(X) + r -1$ elements. We seek a similar statement when $U$ is a singular curve.

Embed $U$ in a connected projective curve $X$ such that $B \coloneqq X - U$ consists of smooth points. Set $r \coloneqq |B|$. Consider the normalization map $\nu \colon X^{\nu} \longrightarrow X$. Then $X^\nu$ is the unique (up to isomorphism) smooth projective connected $k$-curve with function field $k(X) = k(U)$. In general, the pullback divisor $\nu^*(x)$ on $X^{\nu}$ for a closed point $x \in X$ need not be reduced. Let $X^+$ be the seminormal curve obtained from $X^{\nu}$ together with the set of effective reduced divisors $\{\nu^*(x)_{\text{ \rm red}} \coloneqq \nu^{-1}(x)\}_{x \in U_{\rm sing}}$ (see \eqref{eq_sminormal_is_RS}). Since $U$ is connected, $X^+$ is also connected. The map $\nu$ factors as a composition
$$\nu \colon X^\nu \overset{\nu^+} \longrightarrow X^+ \overset{\eta} \longrightarrow X.$$
Then $B$, $B^+ \coloneqq \eta^{-1}(B)$ and $\nu^{-1}(B)$ are all isomorphic sets of $r$-many closed smooth points in the respective curves. We make the following useful note.

\begin{remark}\label{rmk_r}
Set $U^\nu = \nu^{-1}(U)$. The genus $g$ of the curve $X^\nu$, the number $r = |\nu^{-1}(B)| = |X^\nu - U^\nu|$ and the number 
$$\delta \coloneqq \sum_{u \in U_{\text{\rm sing}}} \left( |\nu^{-1}(u)| -1 \right)$$
are all uniquely determined by the curve $U$.
\end{remark}

By construction, $U^+$ is a seminormal affine $k$-curve, and $\eta$ restricts to a finite surjective morphism $\eta|_{U^+} \colon U^+ \longrightarrow U$ which is also subintegral (for each point $u \in U$, there is a unique point in $\eta^{-1}(u)$ having the same residue field). In particular, $\eta|_{U^+}$ is a radical morphism. By \cite[\href{https://stacks.math.columbia.edu/tag/04DF}{Lemma 04DF}]{SP}, it is a universal homeomorphism. By \cite[Theorem~18.1.2, page 110]{EGA}, we have the following isomorphism of \'{e}tale fundamental groups.
\begin{equation}\label{eq_iso_pi}
\pi_1(U) \cong \pi_1(U^+).
\end{equation}

\emph{In whatever follows, we assume that $U$ is seminormal.}

Now suppose that $\pi_1(U) \twoheadrightarrow G$ be a finite continuous quotient. This uniquely corresponds to a $G$-Galois cover $f \colon Y \longrightarrow X$ of connected projective $k$-curves that is \'{e}tale over $U$. Since $X$ is seminormal and the branch locus of $f$ is contained in the smooth locus of $X$, the curve $Z$ is necessarily seminormal. We need a local-global principle (see \cite[Theorem~3.1]{Obus} for a smooth curve) as follows.

\begin{proposition}\label{prop_loc_glob}
Let $R$ be a complete discrete valuation ring with residue field $k$. Under the above notation, let $X_R$ be a connected projective flat $R$-curve with special fibre $X$ obtained in Proposition~\ref{prop_lift_seminormal}. Let $B_R \subset X_R$ be the set of $r$ sections in $X_R(\text{\rm Spec}(R))$ with pairwise disjoint supports corresponding to the set $B$. Suppose that for each point $x \in B$, a subgroup $I_x \subset G$ occurs as an inertia group over $x$. Also assume that for each point $x \in B$, a point $z \in f^{-1}(x)$, the local $I_x$-Galois extension $\widehat{\mathcal{O}}_{Z,z}/\widehat{\mathcal{O}}_{X,x}$ of complete local rings lifts over $\text{\rm Spec}(R)$. Then the cover $f$ lifts to a $G$-Galois cover $f_R \colon Z_R \longrightarrow X_R$ of connected $R$-curves that is \'{e}tale away from $B_R$.
\end{proposition}

\begin{proof}[{\underline{{Sketch of Proof}}}]
The proof is essentially the same as the proof of \cite[Theorem~3.1]{Obus} that uses formal patching. Set $U_R = X_R - B_R$. Since the cover $f$ is \'{e}tale over $U$, by \cite[Proposition~22.5, page 147]{Hartshorne_Def}, there is a $G$-Galois \'{e}tale cover $g_R \colon V_R \longrightarrow U_R$ of connected $R$-curves such that $g_R \times_R k = f|_{U}$. As each point $x \in B$ is a smooth point, $\widehat{\mathcal{O}}_{X,x}$ is a complete discrete valuation ring. By our assumption, there are $I_x$-Galois covers
$$h_x \colon S_x \longrightarrow \text{\rm Spec}\left( \widehat{\mathcal{O}}_{X_R,x} \right)$$
with respective special fibres $\text{\rm Spec}\left( \widehat{\mathcal{O}}_{Z,z} \right) \longrightarrow \text{\rm Spec} \left( \widehat{\mathcal{O}}_{X,x} \right)$. Consider the induced $G$-Galois covers
$$g_x \colon \text{\rm Ind}_{I_x}^G S_x \longrightarrow \text{\rm Spec}\left( \widehat{\mathcal{O}}_{X_R,x} \right).$$
Since the covers $h_x$ are generically \'{e}tale, and $f_R$ is \'{e}tale over $\bigcup_{x \in X} \text{\rm Spec}\left( \widehat{\mathcal{O}}_{X_R,\tilde{x}} \right)$ (here $\tilde{x}$ is the deleted formal neighborhood $\text{\rm Spec}\left( K_{X,x} \right)$ and $K_{X,x} = \text{\rm Fr}(\widehat{\mathcal{O}}_{X,x})$), there is an $G$-equivariant isomorphism
$$V_R \times_{U_R} \text{\rm Spec}\left( \widehat{\mathcal{O}}_{X_R,\tilde{x}} \right) \cong \text{\rm Ind}_{I_x}^G S_x \times_{\text{\rm Spec}\left( \widehat{\mathcal{O}}_{X_R,x} \right)} \text{\rm Spec}\left( \widehat{\mathcal{O}}_{X_R,\tilde{x}} \right)$$
of Galois \'{e}tale covers. We obtain our required $G$-Galois cover $f_R$ by applying \cite[Theorem~3.2.8]{Ha_PG}.
\end{proof}

We come back to our set up where $f \colon Y \longrightarrow X$ is a $G$-Galois cover of connected projective seminormal $k$-curves. Recall from Section~\cref{sec_notation} that the (necessarily normal) subgroup of $G$ generated by the Sylow $p$-subgroups is denoted by $p(G)$. The cover $f$ factors as as composition
$$f \colon Y \xrightarrow[]{p(G)\text{\rm -Galois}} Y' \coloneqq Y/p(G) \xrightarrow[]{G/p(G)\text{\rm -Galois}} X$$
of Galois covers. Moreover, the cover $Y' \longrightarrow X$ is \'{e}tale over $U$, of order prime-to-$p$. In particular, the inertia groups above points in $x \in B$ are cyclic groups of order prime-to-$p$. Then the extensions of complete local rings near each point $x \in B$ lifts over $\text{\rm Spec}(R)$ (\cite[Theorem~1.5, discussion preceding Theorem~3.1]{Obus}). By Proposition~\ref{prop_loc_glob}, there is a $G/p(G)$-Galois cover $Y'_R \longrightarrow X_R$ of connected $R$-curves that is \'{e}tale over $U_R$. Further assume that $K = \text{\rm Fr}(R)$ has characteristic $0$. Without lose of generality, we may assume that $\mathbb{C}$ is an algebraic closure of $K$. The generic fibre $Y'_R \times_R K \longrightarrow X_R \times_R K$ is again $G/p(G)$-Galois, \'{e}tale over $U_R \times_R K$. Its geometric fibre is a $G/p(G)$-Galois cover $Y'_{\mathbb{C}} \longrightarrow X_{\mathbb{C}}$ of connected projective $\mathbb{C}$-curves that is \'{e}tale over $U_{\mathbb{C}} \coloneqq U_R \times_R \mathbb{C}$. As $U$ is irreducible, so is $U_{\mathbb{C}}$. Moreover, $|X_{\mathbb{C}} - U_{\mathbb{C}}| = |B| = r$. Thus we can realize $G/p(G)$ as a (continuous) finite quotient
\begin{equation}\label{eq_quotient}
\pi_1(U_{\mathbb{C}}) \twoheadrightarrow G/p(G).
\end{equation}

\begin{theorem}\label{thm_main_affine}
Let $k$ be any algebraically closed field of arbitrary characteristic and $U$ be an integral affine $k$-curve. Let $X^\nu$ be the unique (up to isomorphism) smooth projective connected $k$-curve with function field $k(X^\nu) = k(U)$. Let $U^\nu \subset X^\nu$ be the maximal affine open subset together with a finite surjective morphism $\nu \colon U^{\nu} \longrightarrow U$ that is an isomorphism over the smooth locus $U_{\text{\rm smooth}} = U - U_{\text{\rm sing}}$. Suppose that $X^\nu$ has genus $g$, \, $r = |X^\nu - U^\nu|$ and set
$$\delta \coloneqq \sum_{u \in U_{\text{\rm sing}}} \left( |\nu^{-1}(u)| -1 \right).$$

A finite group $G$ occurs as a (continuous) quotient of the \'{e}tale fundamental group $\pi_1(U)$ if and only if $G/p(G)$ is generated by at most $2g+r-1+\delta$ elements.
\end{theorem}

\begin{proof}
We may assume that $\text{\rm char}(k) = p >0$ and $U$ is a singular curve (otherwise, the theorem is well known). Observe that the triple $(g,r,\delta)$ in the statement is the same as the triple of numbers in Remark~\ref{rmk_r}. By our previous discussion and the isomorphism~\eqref{eq_iso_pi}, we may assume that $U$ is a seminormal curve embedded in an irreducible projective seminormal curve $X$ such that the finite set $B = X - U$ of $r$ points are all smooth points of $X$. The map $\nu$ in the statement is the restriction of the normalization map $\nu \colon X^\nu \longrightarrow X$.

Suppose that $G$ occurs as a continuous finite quotient of $\pi_1(U)$. Then we have a surjection $\pi_1(U_{\mathbb{C}}) \twoheadrightarrow G/p(G)$ as in~\eqref{eq_quotient} for a similar complex affine curve $U_{\mathbb{C}}$. In particular, $U_{\mathbb{C}}$ is an irreducible seminormal curve which have the same associated triple $(g,r,\delta)$. By the general form of the Riemann Existence Theorem (\cite[Exp XII; Corollary~5.2]{SGA1}), $\pi_1(U_{\mathbb{C}})$ is the profinite completion of the topological fundamental group $\pi_1^{\text{top}}(U_{\mathbb{C}}(\mathbb{C}))$. As the projective curve $X_{\mathbb{C}}$ is obtained from its normalization via identifying two points at a time (see \eqref{eq_factorization}), we conclude that the topological fundamental group $\pi^{\text{top}}_1(U_{\mathbb{C}})$ is a finitely generate free group of the form $\pi^{\text{top}}_1(\nu^{-1}\left(U_{\mathbb{C}}\right)) \ast F_{\delta}$ where $F_{\delta}$ is a finitely generated free group on $\delta$-many generators and we again denote the normalization map $X^\nu_{\mathbb{C}} \longrightarrow X_{\mathbb{C}}$ by $\nu$. Since $\pi^{\text{top}}_1(\nu^{-1}\left(U_{\mathbb{C}}\right))$ is the free group on $2g+r-1$ generators, the forward direction follows.

For the converse, suppose that $G$ is a finite group and $G/p(G)$ has a generating set of size $m$, \, $m \leq 2g+r-1+\delta$. Let $P$ be a Sylow $p$-subgroup of $G$, thus necessarily contained in $p(G)$. Let $N$ be the normalizer subgroup $N = N_G(P)$. Consider the surjective homomorphism $\pi \colon G \twoheadrightarrow G/p(G)$. By \cite[Lemma~5.3]{Ha_AC}, $N$ contains a subgroup $H$ of order prime-to-$p$ such that $\pi(H) = G/p(G)$. In particular, $G$ is generated by $H$ and $p(G)$. Let $a_1, \, \ldots, \, a_m$ be the generators of $G/p(G)$. Choose elements $h_1, \, \ldots, \, h_m \in H$ with $\pi(h_i) = a_i$. Let $m' \coloneqq \text{\rm min}\{m, 2g+r-1\}$. Set $H_1$ to be the subgroup of $H$ generated by $h_1, \, \ldots, \, h_{m'}$. As $m' \leq 2g+r-1$ and $H_1$ is a group of order prime-to-$p$, by \cite[Exp. XIII, Corollary~2.12]{SGA1}, there is a $H_1$-Galois cover $V_1 \longrightarrow X^\nu$ of smooth connected $k$-curves, \'{e}tale away from the $r$ points in $\nu^{-1}(B)$. We apply Proposition~\ref{prop_main_1} inductively to obtain an $H$-Galois cover $V_2 \longrightarrow X$ of connected projective $k$-curves that is \'{e}tale away from the set $B$ of $r$ smooth points. Since $H$ has order prime-to-$p$, the cover $V_2 \longrightarrow X$ is at most tamely ramified. Fix a point $x \in B$, and let $\langle c \rangle \cong \mathbb{Z}/s$ be an inertia group above $x$, \, $(s,p)=1$. Then $c$ normalizes $P$. Consider the semi-direct product $P \rtimes \langle c^{-1} \rangle$. Let $g_1 \colon \mathbb{P}^1 \longrightarrow \mathbb{P}^1$ be any $P \rtimes \langle c^{-1} \rangle$-Galois (Harbater-Katz-Gabber) cover, \'{e}tale away from $\{0,\infty \}$, totally ramified over $\infty$ and $\langle c^{-1} \rangle$ occurs as an inertia group above $0$. By \cite[Theorem~2.1]{Ha_Embed}, we obtain a $\langle \, p(G), P \rtimes \langle \, c^{-1} \, \rangle \, \rangle$-Galois cover $W \longrightarrow \mathbb{P}^1$, \'{e}tale away from $\{0,\infty\}$ such that $P \rtimes \langle \, c^{-1} \, \rangle$ occurs as an inertia group above $\infty$ and $\langle \, c^{-1} \, \rangle$ occurs as an inertia group above $0$. Finally we apply \cite[Corollary to Patching Theorem]{Ha_Patching}\footnote{The same proof works in our set up as the branch point $x$ is assumed to be a smooth point.} to this cover and the $H$-Galois cover $V_2 \longrightarrow X$ to obtain a $G$-Galois cover $Z \longrightarrow X$, \'{e}tale over $U$ and $P \rtimes \langle \, c^{-1} \, \rangle$ occurs as an inertia group above $x$. This uniquely corresponds to a (continuous) surjection $\pi_1(U) \twoheadrightarrow G$.
\end{proof}

\subsection{Tame Fundamental Groups and Inertia Conjecture}\label{sec_IC}
We retain the notation of Section~\cref{sec_Main_Affine}. As in the smooth case, we would also like to get information about the tame fundamental group $\pi_1^t(U \subset X)$ when $X$ is seminormal and $B \coloneqq X - U$ consists of $r$ smooth points. Similar groups are studied in \cite{Saidi} and \cite{Wewers} when $X$ is a semistable curve and $X-U$ consists of singular points.

The finite covers of $X$ that are \'{e}tale over $U$ and tamely ramified over $B$ form a Galois Category $\text{\rm FT}_{U \subset X}$ with the obvious fibre functor, $\pi_1^t(U \subset X)$ is the automorphism group of the fibre functor, and we have an equivalence of categories $\text{\rm FT}_{U \subset X} \overset{\sim} \longrightarrow \text{\rm Finite-}\pi_1^t(U \subset X)\text{\rm -sets}$. For analogous complex curves $U_{\mathbb{C}} \subset X_{\mathbb{C}}$, we also consider the covers of $X_{\mathbb{C}}$ that are \'{e}tale over $U_{\mathbb{C}}$, and the inertia groups of order prime-to-$p$; we have the corresponding category $\text{\rm FT}_{U_{\mathbb{C}} \subset X_{\mathbb{C}}}$ and the group $\pi_1^{p\text{\rm -tame}}(U_{\mathbb{C}} \subset X_{\mathbb{C}})$. The results below are almost similar to what we have obtained so far. To avoid repetition, we give a brief sketch without going into the details.

Suppose that $f \colon Z \longrightarrow X$ is a $G$-Galois cover of connected curves that is \'{e}tale over $U$ and is tamely ramified over $B$. Using the same notation and argument for the Galois cover $f$ in the beginning of this section, we see that $f$ lifts over $\text{\rm Spec}(R)$ where $R$ is any complete discrete valuation ring with residue field $k$, the fraction field $K$ of characteristic $0$, and $\mathbb{C}$ is an algebraic closure of $K$. In particular, we obtain a $G$-Galois cover $f_{\mathbb{C}} \colon Z_{\mathbb{C}} \longrightarrow X_{\mathbb{C}}$ that is \'{e}tale over $U_{\mathbb{C}}$ and tamely ramified over $B_{\mathbb{C}}$. Further, the ramification indices are also the same for $f$ and $f_{\mathbb{C}}$. The construction of the cover also shows that the specialization functor $\text{\rm FT}_{U_R \subset X_R} \longrightarrow \text{\rm FT}_{U \subset X}$ is an equivalence of categories, and the composite functor $\text{\rm FT}_{U \subset X} \overset{\sim}\longrightarrow \text{\rm FT}_{U_R \subset X_R} \longrightarrow \text{\rm FT}_{U_{\mathbb{C}} \subset X_{\mathbb{C}}}$ sends connected objects to connected objects. Thus we also have a surjective specialization map 
\begin{equation}\label{eq_tame_sp}
\text{\rm sp}^t \colon \pi_1^{p\text{\rm -tame}}(U_{\mathbb{C}} \subset X_{\mathbb{C}}) \twoheadrightarrow \pi_1^t(U \subset X).
\end{equation}
From the usual topological argument, we again see that $\pi_1^{p\text{\rm -tame}}(U_{\mathbb{C}} \subset X_{\mathbb{C}})$ is the free product $\pi_1^{p\text{\rm -tame}}(U^\nu_{\mathbb{C}}) \ast \widehat{F_{\delta}}$. Thus $\pi_1^{p\text{\rm -tame}}(U_{\mathbb{C}} \subset X_{\mathbb{C}})$ is a free finitely generated profinite group, and the same is true for the group $\pi_1^t(U \subset X)$. In particular, $\pi_1^t(U \subset X)$ is completely determined by its finite quotients, and if the group has rank $a$, then every finite group on $\leq a$ generators is a finite quotient of it.

Further, we can use the same argument as in the projective case (replacing $\text{\rm FEt}$ by $\text{\rm FT}$ in \eqref{diag_2_comm}) to obtain a surjective homomorphism
$$\pi_1^t(U^\nu) \ast \widehat{F_{\delta}} \twoheadrightarrow \pi_1^t(U \subset X).$$
Finally, using the similar argument for the projective case---Theorem~\ref{thm_main_projective}, we see there is also a surjection in the other direction. As a consequence, we obtain the following structure theorem for the tame fundamental group.

\begin{theorem}\label{thm_tame}
Let $X$ be an irreducible seminormal $k$-curve. Let $U \subset X$ be an affine curve obtained by removing $r$ many closed points from $X$. Let $\nu \colon X^\nu \longrightarrow X$ be the normalization map. Set $U^\nu = \nu^{-1}(U)$. Then the tame fundamental group $\pi_1^t(U \subset X)$ is given as a free product of free finitely generated profinite groups as follows.
$$\pi_1^t(U \subset X) = \pi_1^t(U^\nu) \ast \widehat{F_{\delta}}.$$
\end{theorem}

The same proof of Corollary~\ref{cor_nec}~\ref{item-1} then shows the following.

\begin{corollary}
Under the above notation, the tame specialization map $\text{\rm sp}^t$ in \eqref{eq_tame_sp} induces an isomorphism of the maximal $p'$-quotients.
\end{corollary}

\begin{remark}
We remark that anabelian type conjectures for seminormal $k$-curves can be posed; namely, what information do we get about the curve from its \'{e}tale fundamental group and subsequent quotients? The results from \cite{Tamagawa} for smooth curves then translates to seminormal singular curves once we know how the points in its normalizer are identified.
\end{remark}

To understand the Galois covers of $\mathbb{P}^1_k$, Abhyankar posed the Inertia Conjecture in \cite[Section~16]{Abh_01}. The status of the conjecture remain open at this moment (see \cite{survey_paper}, \cite{Das} and \cite{Das_2} for the progress towards this conjecture and for the generalizations). Since we now know the Galois groups for covers of singular curves, such questions can be posed for singular curves as well. In the following, we treat a particular case. We see that that whole set up for \cite[Section~6]{Das} comes into picture, even in this simple case.

Suppose that $Y = \mathbb{P}^1$, and let $X = Y/\left( 0 {\sim} 1 \right)$ be the seminormal curve obtained from $Y$ by identifying the points $0$ and $1$. By Theorem~\ref{thm_main_affine}, a finite group $G$ occurs as a Galois group for some Galois cover $Z \longrightarrow X$ that is \'{e}tale away from $\infty$ if and only if $G$ is an extension of the normal quasi $p$-subgroup $p(G)$ by a cyclic group. We want to understand the possible inertia groups that occur over $\infty$ for all such covers. We proceed as in \cite[Section~6]{Das}. By \cite[IV, Corollaries 3 and 4]{Serre_loc}, an inertia group $I$ is necessarily of the form $I = P \rtimes \mathbb{Z}/m$ for a $p$-group $P$ and $(m,p)=1$. Suppose that $f \colon Z \longrightarrow X$ be a $G$-Galois cover of connected $k$-curves that is \'{e}tale away from $\infty$. Let $I = P \rtimes \mathbb{Z}/m$ be an inertia group above $\infty$. Consider the (necessarily normal) subgroup $H_1 \coloneqq \langle \, P^G \rangle$, generated by the conjugates of $P$ in $G$. Then $f$ factors as a composition $Z \overset{h}\longrightarrow Z' = Z/H_1 \overset{g_1}\longrightarrow X$. Since $P \subset H_1$, the cover $g_1$ is at most tamely ramified over $\infty$, and \'{e}tale away from $\infty$. If the cover $g_1$ has ramification index $\geq 2$ over $\infty$, the same remains true for the normalized pullback of $g_1$ over $\mathbb{P}^1$. Since $\pi_1^t(\mathbb{A}^1)$ is trivial, this is not possible. So $g_1$ is \'{e}tale everywhere. In particular, $I \subset H_1$. Applying the Riemann Hurwitz theorem for \'{e}tale covers of Gorenstein curves, it also follows that $Z' \cong X$. Further, we can consider the normal subgroup $H_2 = \langle \, P^{H_1} \, \rangle$ of $H_1$ and factorize the cover $h$ as a composition of an $H_2$-Galois cover followed by an \'{e}tale $H/H_1$-Galois cover $X \longrightarrow X$. This way, we obtain a factorization of $f$ as a tower
$$ Z \longrightarrow X \overset{g_l}\longrightarrow X \overset{g_{l-1}}\longrightarrow \cdots \overset{g_1}\longrightarrow X$$
where $g_j$ is an $H_{j-1}/H_j$-Galois \'{e}tale cover, and $H_j$'s are inductively defined as follows.
$$H_0 = G \text{ \rm and } H_{j+1}= \langle \, P^{H_j} \, \rangle.$$
We ask whether the converse holds.

\begin{que}[{\cite[Question~6.2]{Das} for smooth case}]\label{que}
Let $X$, $G$, $I$ and $H_j$'s be as above.
\begin{enumerate}
\item (Stronger Version) Given a tower
$$g \colon X \overset{g_l}\longrightarrow X \overset{g_{l-1}}\longrightarrow \cdots \overset{g_1}\longrightarrow X$$
of $H_{j-1}/H_j$-Galois \'{e}tale covers, does there exit a $G$-Galois cover $Z \longrightarrow X$ dominating $g$ that is \'{e}tale away from $\infty$ and $I$ occurs as as an inertia group above $\infty$?
\item (Weaker Version) Suppose that $p(G)$ is generated by the conjugates of $P$ in $p(G)$. does there exit a $G$-Galois cover $Z \longrightarrow X$, \'{e}tale away from $\infty$, and $I$ occurs as as an inertia group above $\infty$?\label{weak}
\end{enumerate}
\end{que}

We end with a final remark.

\begin{remark}\label{rmk_last}
In the above notation, if there is a connected $p(G)$-Galois cover $f \colon W \longrightarrow \mathbb{P}^1$ that is \'{e}tale away from $\infty$ and $I$ occurs as an inertia group above $\infty$, the Weaker Version of Question~\ref{que}~\eqref{weak} has an affirmative answer. This is because $G$ is generated by $p(G)$ and an element $g \in G$ of prime-to-$p$ order (see \cite[Lemma~5.3]{Ha_AC}), and we can apply the formal patching result \cite[Corollary to Patching Theorem]{Ha_Patching} to the above cover $f$ and any $\langle g \rangle$ \'{e}tale connected cover of $X$ (as mentioned earlier, the proof of the loc. cit. holds in our situation of singular curves; the later $\langle g \rangle$-Galois cover exists by Theorem~\ref{thm_main_projective}). In particular, if the Inertia Conjecture holds for the group $p(G)$, the weaker question has an affirmative answer. This observation applies when $p(G)$ is an Alternating group of $d \geq p$ with $I = P$ a $p$-group (\cite[Theorem~1.7]{DK} and \cite[Theorem~1.4]{Das_2} for $p$ odd; \cite[Theorem~1.6]{Das_2} when $4 \nmid d$ and $p=2$).
\end{remark}

\bibliographystyle{amsplain}

\end{document}